\documentclass[11pt, reqno]{amsart}
\usepackage{mathrsfs}
\usepackage[active]{srcltx}
\usepackage{mathrsfs,amsmath}
\usepackage{mathtools}
\usepackage{longtable}
\usepackage{todonotes}
\usepackage[autostyle]{csquotes}
\usepackage{bbm}
\allowdisplaybreaks

\UseRawInputEncoding

\usepackage{xcolor}
\definecolor{bluecite}{HTML}{0875b7}
\usepackage[unicode=true
bookmarksopen={true},
pdffitwindow=true,
colorlinks=true,
linkcolor=bluecite,
citecolor=bluecite,
urlcolor=bluecite,
hyperfootnotes=false,
pdfstartview={FitH},
pdfpagemode= UseNone]{hyperref}

\newcommand{\ds}{\displaystyle}

\newtheorem{proposition}{Proposition}[section]
\newtheorem{theorem}{Theorem}[section]

\newtheorem{remark}{Remark}[section]
\numberwithin{equation}{section}

\usepackage{bbm}  
\usepackage{dsfont}

\address{\textsc{Alexandru Krist\'aly}:  Department of Economics, Babe\c s-Bolyai University, str. Teodor Mihali 58-60, 400591, Cluj-Napoca, Romania \& Institute of Applied Mathematics, \'Obuda
	University, B\'ecsi \'ut 96, 1034,
	Budapest, Hungary.} 
\email{alexandru.kristaly@ubbcluj.ro; kristaly.alexandru@uni-obuda.hu}
\thanks{Research supported by the
	Excellence Researcher Program \'OE-KP-2-2022 of \'Obuda University, Hungary.}

\subjclass[]{ 
	28A25, 
	26D15,
	46E35, 
	53C23,  53C60, 
	58J60. 
}
\keywords{Sobolev inequality, sharpness,  volume growth, ${\sf CD}(0,N)$ spaces, Riemannian manifolds.}

\title[Volume growths vs. Sobolev inequalities]{Volume growths versus Sobolev inequalities}

\author[Alexandru Krist\'aly]{Alexandru Krist\'aly}

\begin{document}
	\begin{abstract} The paper deals with fine volume growth estimates on metric measures spaces supporting various Sobolev-type inequalities. Given a generic metric measure space, we first prove a quantitative volume growth of metric balls under the validity of a  Sobolev-type  inequality (including Gagliardo--Nirenberg, Sobolev and Nash inequalities, as well as their borderlines, i.e., the  logarithmic-Sobolev, Faber--Krahn, Morrey    and Moser--Trudinger inequalities, respectively),  answering partially a question of Ledoux [\textit{Ann.\ Fac.\ Sci.\ Toulouse Math.}, 2000] in a broader setting. 
		We then prove sharp  Gagliardo--Nirenberg--Sobolev interpolation inequalities -- with their borderlines -- in the setting of metric measure spaces verifying the curvature-dimension condition ${\sf CD}(0,N)$ in the sense of Lott--Sturm--Villani. In addition, the equality cases are also characterized in terms of the $N$-volume cone structure of the ${\sf CD}(0,N)$ space together with the  precise profile of extremizers. 

	\end{abstract}
	\maketitle 
	
		\vspace{-0.7cm}
	
	\section{Introduction and Main results}

Establishing sharp Sobolev-type inequalities on curved spaces  became a central topic of investigations in geometric analysis, initiated by Aubin \cite{Aubin};  these spaces include Riemannian and Finsler manifolds, or even metric measure spaces with certain synthetic curvature restriction.
One of the most important classes of such structures are non-negative Ricci curvature spaces (both classical and synthetic), which appear at the interface of geometry, topology and physics; see e.g.\ Cavalletti and Mondino \cite{CM}, Cheeger and  Gromoll \cite{Cheeger-Gromoll},   Li \cite{Li-Annals}, Lott and Villani \cite{LV}, Ni \cite{Ni-JGA}, Perelman \cite{Perelman},  
Sturm \cite{Sturm-2}. A specific feature of these geometric objects is that they do support sharp Sobolev-type inequalities only in a rigid setting, pointed out first by Ledoux \cite{Ledoux-2}.  

To be more precise, let $(M,g)$ be a complete $n$-dimensional  Riemannian manifold  with non-negative Ricci curvature with its natural distance function ${\sf d}_g$ and canonical measure ${\rm d}v_g$, and assume that $(M,g)$ supports  the \textit{Sobolev  inequality} for some ${\sf C}>0$, i.e.,  
\begin{equation}\label{Sobolev-1}
	\|u\|_{L^q(M,{\rm d}v_g)}\leq {\sf C} \|\nabla_g u\|_{L^p(M,{\rm d}v_g)} ,\ \ u\in C_0^\infty(M),
\end{equation}
where $p\in (1,n)$ and $q=np/(n-p)$.
  Then,  it turns out that the \textit{asymptotic volume ratio} verifies 
\begin{equation}\label{AVR-definicio}
	{\sf AVR}(M):=\lim_{r\to \infty}\frac{V_g(B_x(r))}{\omega_nr^n}\geq \left(\frac{{\sf K}_{\rm opt}}{{\sf C}}\right)^{n},
\end{equation}
see Ledoux \cite{Ledoux-2} and do Carmo and Xia \cite{doCarmo-Xia}, where $V_g(B_x(r))$ stands for the volume of the geodesic ball $B_x(r)=\{y\in M:{\sf d}_g(x,y)<r\}$ in $M$, $\omega_n$ is the volume of the unit ball in $\mathbb R^n$, and ${\sf K}_{\rm opt}$ is the optimal Aubin--Talenti constant in the Euclidean counterpart of \eqref{Sobolev-1}, see Talenti \cite{Talenti-0}. Note that by the Bishop--Gromov volume	comparison principle,  ${\sf AVR}(M)$ is well-defined (does not depend on $x\in M$) and ${\sf AVR}(M)\leq 1$; in particular, if ${\sf C}={\sf K}_{\rm opt}$ in \eqref{Sobolev-1}, then  ${\sf AVR}(M)= 1$, which implies  that $(M,g)$ is isometric to the Euclidean space $\mathbb R^n$, see Ledoux \cite{Ledoux-2}.
The proof of the above result  strongly relies on the Barenblatt  profile of the Talentian function in the Euclidean version of \eqref{Sobolev-1}, i.e.,  $u_{\sf T}(x)=(\lambda+|x|^\frac{p}{p-1})^{{(p-n)/}{p}}$, $x\in \mathbb R^n,$ $\lambda>0$. 
In fact, a careful comparison of certain ODEs and ODIs -- coming from the explicit Barenblatt profile -- combined with the Bishop--Gromov volume comparison on spaces with non-negative Ricci curvature provides  the volume growth estimate \eqref{AVR-definicio} and the subsequent rigidity result.  

At the end of his seminal paper, Ledoux \cite[p.\ 362]{Ledoux-1} conjectured that the validity of \eqref{Sobolev-1} with ${\sf C}={\sf K}_{\rm opt}$(=optimal Aubin--Talenti constant) in a Riemannian manifold $(M,g)$  without \textit{any} curvature assumption, still implies the  volume growth  
$V_g(B_x(r))\geq \omega_n r^n$ for every $x\in M$ and $r>0$. Note that in the limit case $p=1$ the conjecture holds, when \eqref{Sobolev-1} reduces to the isoperimetric inequality. Ledoux's conjecture has been 'asymptotically' solved by Carron \cite{Carron}, stating that if the limit $L:=\lim_{r\to \infty}\frac{V_g(B_{x_0}(r))}{\omega_nr^n}$ exists form some $x_0\in M$, the validity of \eqref{Sobolev-1} implies  $L\geq \left(\frac{{\sf K}_{\rm opt}}{{\sf C}}\right)^{n};$ this proof  explores again the explicit form of the Barenblatt  profile of the Talentian bubble $u_{\sf T}$. 

Motivated by the works of Ledoux \cite{Ledoux-1} and Carron \cite{Carron}, the first purpose of our paper is to prove that a general class of Sobolev-type inequalities  implies a quantitative volume growth of metric balls on not necessarily smooth metric measure spaces with no curvature restriction, even in the case when \textit{no} Barenblatt profile occurs in the model/comparison setting, associated with the initial Sobolev inequality. Such situations appear for instance in the case of the Nash inequality, or the borderline cases of the Gagliardo--Nirenberg--Sobolev inequality, as the Moser--Trudinger and $L^p$-Faber--Krahn-type inequalities, respectively.

In order to fix the ideas, let  $N>1$ be a real number, and consider the $1$-dimensional model metric measure cone  $(\mathbb R_+,|\cdot|,{\sf m}_{N}),$ where $\mathbb R_+=[0,\infty) $, $|\cdot|$ is the usual distance on $\mathbb R_+$ and $ {\sf m}_{N}=N\omega_N r^{N-1}\mathcal L^1$ is a weighted measure on $\mathbb R_+$, with $\omega_N=\pi^{N/2}/\Gamma(N/2+1).$ 
For some parameters $q,r>0$, $p>1$  and $\theta\in (0,1]$, verifying the property
\begin{equation}\label{balance-condition}
	\frac{1}{q}=\theta\left(\frac{1}{p}-\frac{1}{N}\right)+\frac{1-\theta}{r},
\end{equation}
 we assume  the validity of a sharp Gagliardo--Nirenberg--Sobolev  inequality on $(\mathbb R_+,|\cdot|,{\sf m}_{N})$,  i.e.,  
\begin{equation}\label{GN-1-dimensional-model}
	\|u\|_{L^q(\mathbb R_+,{\sf m}_{N})}\leq {\sf K}_{\rm opt} \|u'\|_{L^p(\mathbb R_+,{\sf m}_{N})}^\theta \|u\|_{L^r(\mathbb R_+,{\sf m}_{N})}^{1-\theta},\ \ \forall u\in W_{\rm loc}^{1,p}(\mathbb R_+,{\sf m}_N)\cap L^q(\mathbb R_+,{\sf m}_N),
\end{equation}
and the existence of an extremal $C^1$-function $u_0\geq 0$ in \eqref{GN-1-dimensional-model} having the support supp$(u_0)=[0,R_0]$ for some $R_0>0$ (with  supp$(u_0)=\mathbb R_+$ if $R_0=\infty$),  together with the  properties  
\begin{equation}\label{extremal-asymptotics}
\left\{ 
\begin{array}{lll}
	|u_0'|^p\in BV_{\rm loc}([0,R_0)); \\ 
	\exists i_0>0\ {\rm such\ that}\ u_0''(\rho)\geq 0\ {\rm for\ a.e.}\ \rho>i_0, {\rm whenever}\ R_0=\infty; \\
	u_0^q(\rho)\rho^N\to 0, |u_0'|^p(\rho)\rho^N\to 0\ {\rm and} \ u_0^r(\rho)\rho^N \to 0\  {\rm as}\  \rho\to \infty, {\rm whenever}\ R_0=\infty.
%
\end{array}\right.
\end{equation}
The seemingly involved assumptions in \eqref{extremal-asymptotics} serve as  integrability conditions for the extremizer $u_0$ of \eqref{GN-1-dimensional-model}, whose validity is well-known for various range of parameters; for details, see  \S \ref{section-5}. Inequality \eqref{GN-1-dimensional-model} requires the (scale-invariant) balance condition  \eqref{balance-condition}, and it corresponds to the 'radial' version of \eqref{Sobolev-1} in $\mathbb R^N$ when $N=n\in \mathbb N$ and $\theta=1$. Clearly, the last assumption in \eqref{extremal-asymptotics} involving the parameter $r>0$ is not considered usually whenever $\theta=1$.

Let  $(X,{\sf d},{\sf m})$ be a  metric measure space,  $B_x(r)=\{y\in X:{\sf d}(x,y)<r\}$ be the metric ball with origin  $x\in X$ and radius $r>0$.  
Our first result roughly says that if a sharp Sobolev-type inequality holds on the 1-dimensional model space  $(\mathbb R_+,|\cdot|,{\sf m}_{N})$ having  a non-zero extremal function and   $(X,{\sf d},{\sf m})$  supports the same type of Sobolev inequality, then there is a precise volume growth control of the metric balls of $X$ 'at infinity'. To state it,   $\|\nabla u\|_{L^p(X,{\sf m})}$ stands for the $p$-Cheeger energy of  a suitable function $u:X\to \mathbb R,$ while $W_{\rm loc}^{1,p}(X,{\sf m})$ is a local Sobolev space; for details, see \S \ref{section-Cheeger}.  More precisely, we have: 
	
	\begin{theorem}\label{main-theorem-unified} {\rm (Unified volume growth)} Let $q,r>0$,  $p,N>1$ and $\theta\in (0,1]$ be parameters  verifying \eqref{balance-condition} and assume the following statements hold$:$ 
		\begin{enumerate}
			\item[(i)] The Gagliardo--Nirenberg--Sobolev  inequality \eqref{GN-1-dimensional-model} holds on $(\mathbb R_+,|\cdot|,{\sf m}_{N})$ with the optimal constant ${\sf K}_{\rm opt}$,  having  a non-zero and non-increasing extremal $C^1$-function $u_0\geq 0$ on  {\rm supp}$(u_0)=[0,R_0]$ $($with  {\rm supp}$(u_0)=\mathbb R_+$ whenever $R_0=\infty$$)$, verifying \eqref{extremal-asymptotics}$;$ 
			\item[(ii)] $(X,{\sf d},{\sf m})$ is a metric measure space 
			supporting the Gagliardo--Nirenberg--Sobolev  inequality  for some ${\sf C}>0$, i.e., 
			\begin{equation}\label{GN-metric-space}
				\|u\|_{L^q(X,{\sf m})}\leq {\sf C} \|\nabla u\|_{L^p(X,{\sf m})}^\theta \|u\|_{L^r(X,{\sf m})}^{1-\theta},\ \ \forall u\in  W_{\rm loc}^{1,p}(X,{\sf m})\cap L^q(X,{\sf m}); 
			\end{equation}
		\item[(iii)] For some $x_0\in X$ the following limit exists$:$ 
		\begin{equation}\label{Limit-exists-N}
				L_N(x_0)=\lim_{r\to \infty}\frac{	{\sf m}(B_{x_0}(r))}{\omega_N r^{N}}.
		\end{equation}
		\end{enumerate}
	 Then 
		 \begin{equation}\label{volume-growth-estimate}
		 	{\sf AVR}(X):= L_N(x_0)\geq \left(\frac{{\sf K}_{\rm opt}}{\sf C}\right)^\frac{N}{\theta}.
		 \end{equation}
	\end{theorem}
We note that whenever the limit $L_N(x_0)$ in \eqref{Limit-exists-N} exists for some $x_0\in X$,  one can easily prove that  $L_N(x)=L_N(x_0)$ for \textit{every} $x\in X$; in this way, we may consider this common limit 
\begin{equation}\label{AVR-definicio-metric-space}
	{\sf AVR}(X):=L_N(x_0)=\lim_{r\to \infty}\frac{{\sf m}(B_{x_0}(r))}{\omega_Nr^N}
\end{equation}
to be the \textit{asymptotic volume ratio} on the metric measure space  $(X,{\sf d},{\sf m})$.

 The volume growth estimate \eqref{volume-growth-estimate} is comparable with \eqref{AVR-definicio}; note however that while \eqref{AVR-definicio} has been obtained under a curvature condition (i.e., the Ricci curvature is non-negative), the estimate \eqref{volume-growth-estimate} does not require \textit{any} curvature restriction. Thus, in the same spirit as in Carron \cite{Carron}, Theorem \ref{main-theorem-unified} answers 'asymptotically' the question posed by Ledoux \cite{Ledoux-1} on not necessarily smooth spaces for a \textit{broad} family of Sobolev inequalities. In fact, although Theorem \ref{main-theorem-unified} is stated for the Gagliardo--Nirenberg--Sobolev  inequality -- including e.g.\ Nash and Sobolev inequalities  -- our argument works also for its borderline cases, as the logarithmic-Sobolev, Faber--Krahn, Morrey and Moser--Trudinger inequalities; see \S \ref{section-5} \& \ref{section-borderline-0} for details.  
 We also point out that our argument does not even require an explicit shape of the extremal function in  the  Sobolev-type inequality  on the model space, which is useful e.g.\ for the $L^p$-Faber--Krahn and  Moser--Trudinger inequalities,  where the existence of  extremal functions  is known without their explicit shapes.

 The proof of Theorem \ref{main-theorem-unified}, together with the borderline cases of the Gagliardo--Nirenberg--Sobolev  inequality,  
  explores the existence of extremal functions in the model space $(\mathbb R_+,|\cdot|,{\sf m}_{N})$, combined with a  change of variables formula (see Proposition \ref{proposition-co-area})  and a careful blow-down limiting argument. Instead of \eqref{volume-growth-estimate}, one can provide  a more involved control on the volume of metric balls whenever the limit \eqref{Limit-exists-N} does not exist, see Remark \ref{remark-b}/(b). 
  


An efficient application of Theorem \ref{main-theorem-unified} is the proof of a whole class of sharp Sobolev-type inequalities in  metric measure spaces verifying the curvature-dimension condition ${\sf CD}(0,N)$ in the sense of Lott--Sturm--Villani, together with the characterization of the equality cases. Note that  ${\sf CD}(0,N)$ spaces include not only Riemannian and Finsler manifolds with non-negative Ricci curvature (having at most dimension $N$), but also their Gromov--Hausdorff limits with possible singularities,  Alexandrov spaces with non-negative
curvature, etc. 
We emphasize that on ${\sf CD}(0,N)$ spaces, due to the generalized Bishop--Gromov principle, see Sturm \cite{Sturm-2}, the limit $L_N(x)$ in \eqref{Limit-exists-N} exists for every $x\in X$; thus ${\sf AVR}(X)$ in \eqref{AVR-definicio-metric-space} is well-defined. 
Our second main result reads as follows:

	\begin{theorem}\label{main-theorem-CD}  {\rm (Sharp Gagliardo--Nirenberg--Sobolev inequalities on ${\sf CD}(0,N)$ spaces)} Let  $q,r>0$,  $p,N>1$  and $\theta\in (0,1]$ be  para\-meters verifying \eqref{balance-condition}, and let $(X,{\sf d},{\sf m})$ be an essentially non-branching  ${\sf CD}(0,N)$ metric measure space with ${\sf AVR}(X) >0.$ If the Gagliardo--Nirenberg--Sobolev interpolation inequality \eqref{GN-1-dimensional-model} holds on $(\mathbb R_+,|\cdot|,{\sf m}_{N})$ with the optimal constant ${\sf K}_{\rm opt}$ and having  a non-zero, non-increasing extremal 
		$C^1$-function $u_0\geq 0$ on  {\rm supp}$(u_0)=[0,R_0]$ $($with  {\rm supp}$(u_0)=\mathbb R_+$ whenever $R_0=\infty$$)$, verifying \eqref{extremal-asymptotics}$,$
		 then 
		one has that 
			\begin{equation}\label{GN-CD}
			\|u\|_{L^q(X,{\sf m})}\leq {\sf AVR}(X)^{-\frac{\theta}{N}} {\sf K}_{\rm opt}\|\nabla u\|_{L^p(X,{\sf m})}^\theta \|u\|_{L^r(X,{\sf m})}^{1-\theta},\ \ \forall u\in W_{\rm loc}^{1,p}(X,{\sf m})\cap L^q(X,{\sf m}), 
		\end{equation}
		and the constant  ${\sf AVR}(X)^{-\frac{\theta}{N}} {\sf K}_{\rm opt}$ is optimal. 
		
		In addition, if the extremal function $u_0$ in  \eqref{GN-1-dimensional-model}  is unique $($up to scaling and multiplicative factor$)$ and $u_0'\neq 0$ a.e.\ on {\rm supp}\,$(u_0)$, then the following two statements are equivalent$:$ 
		\begin{enumerate}
			\item[(i)] Equality holds in \eqref{GN-CD} for some non-zero function $u\in W_{\rm loc}^{1,p}(X,{\sf m})\cap L^q(X,{\sf m});$
				\item[(ii)] $X$ is an $N$-volume cone with a tip $x_0\in X$ and up to scaling and multiplicative factor,  $u(x)=u_0\left({\sf AVR}(X)^{\frac{1}{N}}{\sf d}(x_0,x)\right)$  for ${\sf m}$-a.e.\   $x\in B_{x_0}(R_0/{\sf AVR}(X)^{\frac{1}{N}})$.
		\end{enumerate}
\end{theorem}

The key ingredient in the proof of Theorem \ref{main-theorem-CD} is the non-smooth P\'olya--Szeg\H o inequality together with the investigation of the equality case,  studied by Nobili and Violo \cite{NV-new};  this result uses a suitable rearrangement from the metric measure space $(X,{\sf d},{\sf m})$ to the 1-dimensional model cone $(\mathbb R_+,|\cdot|,{\sf m}_{N})$, the sharp isoperimetric inequality on ${\sf CD}(0,N)$ spaces from \cite{BK}, and the co-area formula. The sharpness of \eqref{GN-CD}  directly follows by Theorem \ref{main-theorem-unified}. 
	As in Theorem \ref{main-theorem-unified}, one can discuss the limit cases of \eqref{GN-CD}, i.e., logarithmic-Sobolev   and Moser--Trudinger inequalities, see \S \ref{section-borderline-0}.  
	
In order to emphasize the usefulness of Theorem \ref{main-theorem-CD}, we state, as a simple consequence, the sharp Nash inequality on Riemannian manifolds with non-negative Ricci curvature, which seems to be new in the literature.  
We recall that for $p=q=2$, $r=1$ and $\theta=\frac{n}{n+2}$, $n\geq 2$, the Gagliardo--Nirenberg inequality \eqref{GN-1-dimensional-model} reduces to the famous (radial) \textit{Nash inequality}, the
best constant 
 being given by Carlen and Loss \cite{Carlen-Loss}, i.e., 
\vspace{-0.3cm}
\begin{equation}\label{Nash-constant}
	{\sf CL}_n=\left(\frac{n+2}{2}\right)^\frac{1}{2}\omega_n^{-\frac{1}{n+2}}\left(\frac{n}{2}j^2_\frac{n}{2}\right)^{-\frac{n}{2(n+2)}},
\end{equation}
where $j_\nu$ stands for the first positive root of the Bessel function $J_\nu$ of order $\nu$, while the unique 
extremal function -- up to scaling and multiplicative factor -- has compact support, expressed in terms of Bessel functions (thus, not having a Barenblatt profile). 


\begin{theorem}\label{main-theorem-Nash} {\rm (Sharp Nash inequality)} Let  $(M,g)$ be an $n$-dimensional complete  Riemannian manifold with non-negative Ricci curvature and ${\sf AVR}(M)>0$, and let $\nu=\frac{n}{2}-1\geq 0$. Then 
	\begin{equation}\label{Nash-inequality-AVR}
		\displaystyle  	\|u\|_{L^2(M,{\rm d}v_g)}\leq {\sf AVR}(M)^{-\frac{1}{n+2}} {\sf CL}_n \|\nabla_g u\|_{L^2(M,{\rm d}v_g)}^{\frac{n}{n+2}} \|u\|_{L^1(M,{\rm d}v_g)}^{\frac{2}{n+2}},\ \ \forall  u\in  W^{1,2}(M),
	\end{equation} 
	and the constant ${\sf AVR}(M)^{-\frac{1}{n+2}} {\sf CL}_n$  is sharp. 
	
	Moreover,  equality holds in \eqref{Nash-inequality-AVR} for some $u\in W^{1,2}(M)\setminus\{0\}$ if and only if $(M,g)$ is isometric to the Euclidean space $\mathbb R^n,$ and  up to scaling and a multiplicative factor, for some   $x_0\in M$ one has 
	\begin{equation}\label{u-Nash-extrem}
		u(x)=\ds\left\{ 
		\begin{array}{lll}
			1-\frac{1}{J_{\nu}(j_{\nu+1})}{{\sf d}_g(x_0,x)}^{-\nu} J_{\nu}\left(j_{\nu+1}{{\sf d}_g(x_0,x)}\right)\ &\text {if}&  x\in B_{x_0}(1); \\
			
			0 &\text {if}&  x\notin B_{x_0}(1).
		\end{array}\right.
	\end{equation}
\end{theorem}

In particular, Theorem \ref{main-theorem-Nash} directly implies the rigidity result of Druet, Hebey and  Vaugon \cite{DHV}: if the Nash inequality on $(M,g)$ holds with the constant ${\sf CL}_n$, by the sharpness of \eqref{Nash-inequality-AVR}, one has that ${\sf AVR}(M)^{-\frac{1}{n+2}} {\sf CL}_n\leq {\sf CL}_n$, i.e., ${\sf AVR}(M)\geq 1$,  which implies that $(M,g)$ is isometric to $\mathbb R^n.$

The paper is organized as follows. In Section \ref{section-prel} we recall those notions and results that are needed in the paper. 
In Section \ref{section-unified} we prove the unified volume growth estimate on metric measure spaces, see Theorem \ref{main-theorem-unified}. Section \ref{section-4} is devoted to the proof of Theorem \ref{main-theorem-CD}. In Section \ref{section-5} we 
 discuss some particular cases of the Gagliardo--Nirenberg--Sobolev  inequality (Sobolev, Nash, Morrey and Faber--Krahn inequalities), while in Section \ref{section-borderline-0} we deal with its genuine limit cases, i.e., logarithmic-Sobolev and   Moser--Trudinger inequalities. Finally, Section \ref{section-final} is devoted to some concluding remarks.

\section{Preliminaries}\label{section-prel}

\subsection{Metric measure spaces, change of variables \& $p$-Cheeger energy}\label{section-Cheeger}

Let  $(X,{\sf d}, {\sf m})$ be a metric measure space, i.e.,
$(X,{\sf d})$ is a complete separable metric space and 
$ {\sf m}$ is a locally finite measure on $X$ endowed with its
Borel $\sigma$-algebra. We assume that supp$({\sf m})=X$. 

For  every $p >0$ and open set $\emptyset\neq \Omega\subseteq X$, we denote by $$L^p(\Omega,{\sf m})=\left\{ u\colon \Omega\to\mathbb{R}: u\text{ is measurable, }\displaystyle\int_\Omega\vert u\vert^p\,{\rm d}{\sf m}<\infty\right\}, $$ the set of $p$-integrable functions over $\Omega$;  functions in the latter definition which are equal {\sf m}-a.e.\ are  identified. The norm on $L^p(\Omega,{\sf m})$ is denoted by $\|\cdot\|_{L^p(\Omega,{\sf m})}$. When ${\sf m}=\mathcal L^n$ is the $n$-dimensional Lebesgue measure, we simply use  $L^p(\Omega)$ instead of $L^p(\Omega,\mathcal L^n)$ for every open set $\emptyset\neq \Omega\subset \mathbb R^n.$

Through the paper, we need the following change of variables formula on  metric measure spaces. 

\begin{proposition}\label{proposition-co-area} Let $(X,{\sf d},{\sf m})$ be a metric measure space, $x_0\in X$, $R>0$ $($possibly $R=+\infty$$)$ and  $h:[0,R)\to \mathbb R$ be a continuous, locally BV-function such that $h(t){\sf m}(B_{x_0}(t))=\mathcal O(1)$ as $t\nearrow R$ and $h'(t)\,{\sf m}(B_{x_0}(t))\in L^1([0,R))$. 
	If $\rho_{x_0}={\sf d}(x_0,\cdot)$ is the distance  from $x_0\in X$, then $h\circ \rho_{x_0}\in L^1(B_{x_0}(R))$ and 
	\begin{equation}\label{layer-cake}
		\int_{B_{x_0}(R)} h\circ \rho_{x_0}\, {\rm d}{\sf m}=\lim_{t\nearrow R} h(t){\sf m}(B_{x_0}(t))  -\int_0^R h'(t)\,{\sf m}(B_{x_0}(t)){\rm d}t.
	\end{equation}
\end{proposition}

\begin{proof} Let $t_0\in (0,R)$ be arbitrarily fixed. Since $h$ is a BV-function in $[0,t_0]$, its Jordan decomposition provides two non-increasing functions $h_1,h_2:[0,t_0]\to \mathbb R$ such that $h=h_1-h_2$ on $[0,t_0]$. Without loss of generality, we may assume that $h_i,$ $i\in \{1,2\}$, are non-negative on $[0,t_0]$. 
	By the layer cake representation and a change of variable, one has for $i\in \{1,2\}$ that 
	\begin{align*}
		\int_{B_{x_0}(t_0)} h_i\circ \rho_{x_0}\, {\rm d}{\sf m}&=\int_0^\infty{\sf m}\left(\{x\in B_{x_0}(t_0):h_i\circ \rho_{x_0}(x)>s\}\right){\rm d}s\\&=\int_{0}^{h_i(t_0)}{\sf m}\left(\{x\in B_{x_0}(t_0):h_i({\sf d}(x_0,x))>s\}\right){\rm d}s\\&\ \ \ \ +\int_{h_i(t_0)}^{h_i(0)}{\sf m}\left(\{x\in B_{x_0}(t_0):h_i({\sf d}(x_0,x))>s\}\right){\rm d}s\ \ \ \  \ \ [s=h_i(t)] \\&=h_i(t_0){\sf m}(B_{x_0}(t_0))+\int_{t_0}^0 h_i'(t){\sf m}(B_{x_0}(t)){\rm d}t.
	\end{align*}
	Therefore, one has that 
	$$\int_{B_{x_0}(t_0)} h\circ \rho_{x_0}\, {\rm d}{\sf m}=h(t_0){\sf m}(B_{x_0}(t_0))-\int_0^{t_0} h'(t){\sf m}(B_{x_0}(t)){\rm d}t.$$		
	Letting $t_0\to R$ and using the assumptions $h(t){\sf m}(B_{x_0}(t))=\mathcal O(1)$ as $t\nearrow R$ and $h'(t)\,{\sf m}(B_{x_0}(t))\in L^1([0,R))$, relation \eqref{layer-cake} follows at once. \end{proof}

Let ${\rm Lip}(\Omega)$  (resp. ${\rm Lip}_{bs}(\Omega)$, ${\rm Lip}_c(\Omega)$, and ${\rm Lip}_{\rm loc}(\Omega)$) be the space of real-valued Lipschitz  (resp.\ boundedly supported Lipschitz,
 compactly  supported Lipschitz,  and  locally Lipschitz) functions over $\Omega$. 
 If $u\in {\rm Lip}_{\rm loc}(X)$,  the \textit{local Lipschitz constant} $|{\rm lip}_{\sf d} u|(x)$ of $u$ at $x\in X$  is 
$$|{\rm lip}_{\sf d} u|(x)=\limsup_{y\to x}\frac{|u(y)-u(x)|}{{\sf d}(x,y)}.$$

 Following  Ambrosio,  Gigli and Savar\'e \cite{AGS} and  Cheeger \cite{Cheeger},  we introduce the Sobolev space over the metric measure space $(X,{\sf d}, {\sf m})$. To do this, let $p>1$. The $p$-\textit{Cheeger energy} ${\sf Ch}_p:L^p(X,{\sf m})\to [0,\infty]$ is defined as the convex and lower semicontinuous functional 
$${\sf Ch}_p(u)=\inf\left\{\liminf_{n\to \infty} \int_X |{\rm lip}_{\sf d} u_n|^p {\rm d} {\sf m}:(u_n)\subset {\rm Lip}(X)\cap L^p(X,{\sf m}),\ u_n\to u\ {\rm in}\ L^p(X,{\sf m})  \right\}.$$
Then  $$W^{ 1,p}(X,{\sf m})=\{u\in L^p(X,{\sf m}):{\sf Ch}_p(u)<\infty\}$$ is  the $p$-Sobolev space over $(X,{\sf d}, {\sf m})$, endowed with the norm 
$\|u\|_{W^{ 1,p}}=\left(\|u\|^p_{L^p(X,{\sf m})}+{\sf Ch}_p(u)\right)^{1/p}. $
Note that $W^{ 1,p}(X,{\sf m})$ is a Banach space. By the relaxation of the $p$-Cheeger energy,  one can define the minimal $ {\sf m}$-a.e.\ object $|\nabla u |_p \in  L^p(X, {\sf m})$, the so-called  \textit{minimal $p$-weak upper gradient of} 
$u \in L^p(X,{\sf m})$ 
such that 
$${\sf Ch}_p(u)= \int_X |\nabla u|_p^p\, {\rm d} {\sf m}.$$
When no confusion arises, we shall write $\|\nabla u\|_{L^p(X,{\sf m})}$ instead of ${\sf Ch}_p^{1/p}(u)$. 

Given $u\in L_{\rm loc}^p(\Omega)$, we say that $u\in W_{\rm loc}^{1,p}(\Omega,{\sf m})$ whenever $\eta u\in W^{1,p}(X,{\sf m})$ for every $\eta\in {\rm Lip}_{bs}(\Omega)$ (assuming $\eta u$ is extended by zero in  $X\setminus \Omega$) and we set $|\nabla u|_p:=|\nabla (\eta u)|_p$ ${\sf m}$-a.e. on $\{\eta=1\}$. We notice that the term $|\nabla u|_p$ introduced in this way is well-defined, see Gigli and Pasqualetto \cite{GP}. 
According to  Cheeger \cite{Cheeger}, if $u\in {\rm Lip}_{\rm loc}(X)$,   one has that 
\begin{equation}\label{Cheeger-0-ineq}
	|\nabla u|_p\leq |{\rm lip}_{\sf d}u| \ \ {\sf m}{\rm -a.e.\ on}\  X.
\end{equation}  

\subsection{${\sf CD}(0,N)$ spaces, rearrangement \& P\'olya--Szeg\H o inequality} \label{section-rearrangement}


Let	$P_2(X,{\sf d})$ be the
$L^2$-Wasserstein space of probability measures on $X$, and
$P_2(X,{\sf d}, {\sf m})$ be the subspace of
$ {\sf m}$-absolutely continuous measures on $X$.
Given  $N> 1,$ let 
${\rm Ent}_N(\cdot| {\sf m}):P_2(X,{\sf d})\to \mathbb R$ be  
the {\it R\'enyi entropy functional} with
respect to the measure $ {\sf m}$, defined by 
\begin{equation}\label{entropy}
	{\rm Ent}_N(\nu| {\sf m})=-\int_X \rho^{-\frac{1}{N}}{\rm d}\nu=-\int_X \rho^{1-\frac{1}{N}}{\rm d} {\sf m},
\end{equation}
where $\rho$ is the density function of $\nu^{\rm ac}$ in the decomposition 
$\nu=\nu^{\rm ac}+\nu^{\rm s}=\rho  {\sf m}+\nu^{\rm s}$, where $\nu^{\rm ac}$ and $\nu^{\rm s}$
stand for the absolutely continuous and singular parts of $\nu\in
P_2(X,{\sf d}),$ respectively.

The \textit{curvature-dimension condition} ${\sf CD}(0,N)$
states that for all $N'\geq N$ the functional ${\rm Ent}_{N'}(\cdot|\,{\sf m})$ is
convex on the $L^2$-Wasserstein space $P_2(X, {\sf d}, {\sf m})$, i.e.,  for each
$ {\sf m}_0, {\sf m}_1\in  P_2(X,{\sf d}, {\sf m})$ there exists
a geodesic
$\Gamma:[0,1]\to  P_2(X,{\sf d}, {\sf m})$ joining
$ {\sf m}_0$ and $ {\sf m}_1$ such that for every $s\in [0,1]$ one has 
$${\rm Ent}_{N'}(\Gamma(s)| {\sf m})\leq (1-s) {\rm Ent}_{N'}( {\sf m}_0| {\sf m})+s {\rm Ent}_{N'}( {\sf m}_1| {\sf m}),$$
see Lott and Villani \cite{LV} and Sturm \cite{Sturm-2}.

In the sequel, let $(X,{\sf d}, {\sf m})$ be a ${\sf CD}(0,N)$ space for some $N>1$. Due to Sturm \cite{Sturm-2},  the Bishop--Gromov comparison principle states that 
$$
r\mapsto \frac{ {\sf m}(B_x(r))}{r^{N}},\ \ r>0,$$
is non-increasing on $[0,\infty)$ for every $x\in X$. In particular, the \textit{asymptotic volume ratio}
$${\sf AVR}(X)=\lim_{r\to \infty}\frac{ {\sf m}(B_x(r))}{\omega_Nr^N},$$
is well-defined, i.e., it exists and is independent of the choice of $x\in X$.

If ${\sf AVR}(X)>0 $, we have the following \textit{isoperimetric inequality} on $(X,{\sf d}, {\sf m})$: for  every bounded Borel subset $\Omega\subset X$ one has 
\begin{equation}\label{eqn-isoperimetric}
	{\sf m}^+(\Omega)\geq N\omega_N^\frac{1}{N}{\sf AVR}(X)^\frac{1}{N} {\sf m}(\Omega)^\frac{N-1}{N},
\end{equation}
and the constant $N\omega_N^\frac{1}{N}{\sf AVR}_ {\sf m}^\frac{1}{N}$ in  \eqref{eqn-isoperimetric} is sharp, see Balogh and Krist\'aly \cite{BK}. Here,
\begin{equation}\label{minkowski-content}
	{ {\sf m}}^+(\Omega)=\liminf_{\varepsilon \to 0^+}\frac{ {\sf m}(\Omega_\varepsilon\setminus \Omega)}{\varepsilon},
\end{equation}
stands for the \textit{Minkowski content} of $\Omega\subset X$,
where $\Omega_\varepsilon=\{x\in X:\exists\, y\in \Omega\ {\rm such\ that}\ {\sf d}(x,y)< \varepsilon\}$ is the $\varepsilon$-neighborhood of $\Omega$ w.r.t.\  the metric ${\sf d}$. 
 Inequality \eqref{eqn-isoperimetric} has been stated  on Riemannian manifolds by Agostiniani,  Fogagnolo and  Mazzieri \cite{AFM, Fogagnolo-Maz-JFA} in low dimensions (up to dimension 7) and Brendle \cite{Brendle} in any dimension. 

 Since $(X,{\sf d}, {\sf m})$ is a ${\sf CD}(0,N)$ space, we know by Rajala \cite{Rajala}    that $(X,{\sf d}, {\sf m})$ is doubling (by Bishop--Gromov comparison principle) and  supports the $(1,p)$-Poincar\'e inequality for every $p\geq 1$. As a consequence of these results, it follows that $|\nabla u |_p$
 is independent of $p>1$, denoted simply as $|\nabla u |,$   see Cheeger \cite{Cheeger} and Gigli \cite{Gigli}, and instead of \eqref{Cheeger-0-ineq},  for every $u\in {\rm Lip}_{\rm loc}(X)$ one has  
\begin{equation}\label{Cheeger-equality}
	|\nabla u |=|{\rm lip}_{\sf d}u| \ \ {\sf m}{\rm -a.e.\ on}\  X,
\end{equation}
see Cheeger \cite[Theorem 6.1]{Cheeger}. 
In addition, by \eqref{Cheeger-equality}, one has  the eikonal equation
\begin{equation}\label{eikonal}
	|\nabla \rho_{x_0}|=1\ \ {\sf m}{\rm -a.e.\ on }\ X,
\end{equation}
where $\rho_{x_0}={\sf d}(x_0,\cdot)$;  
see e.g.\  Gigli \cite[Theorem 5.3]{Gigli}.

For a measurable function $u\colon X\to[0,\infty)$, let $u^*\colon[0,\infty)\to[0,\infty)$ be the \textit{non-increasing rearrangement} of $u$ defined on the 1-dimensional model space $(\mathbb R_+,\vert\cdot\vert,{\sf m}_N=N\omega_Nr^{N-1}\mathcal L^1)$ so that
\begin{align}\label{rearrangement}
	{\sf m}\big(M_t(u)\big)={\sf m}_N(M_t(u^*)),\ \ \forall t>0,
\end{align}
where
\begin{align*}
	M_t(u)=\lbrace x\in X: u(x)>t\rbrace\;\text{ and }\;M_t(u^*)=\lbrace y\in[0,\infty):u^*(y)>t\rbrace,
\end{align*}
whenever ${\sf m}\big(M_t(u)\big)<\infty$ for every $t>0$. 
In particular, the open set $\Omega\subset X$ can be rearranged into the interval $\Omega^*=[0,r)$ with $ {\sf m}(\Omega)={\sf m}_N([0,r))=\omega_N r^N,$ with the convention that $\Omega^*=[0,\infty)$ whenever $ {\sf m}(\Omega)=+\infty$. By the layer cake representation, one has for every continuous function $F\colon[0,\infty)\to \mathbb R$ and $u\colon X\to[0,\infty)$  the \textit{Cavalieri principle}: 
\begin{equation}\label{Cavallieri}
\int_\Omega	F(u){\rm d}{\sf m}= \int_{\Omega^*}	F(u^*){\rm d}{\sf m}_N,
\end{equation}
assuming that the integrals are well-defined. 

Since $(X,{\sf d}, {\sf m})$ verifies the  ${\sf CD}(0,N)$ condition, by the isoperimetric inequality \eqref{eqn-isoperimetric} and the co-area formula (see e.g. Miranda \cite{Miranda}, Mondino and Semola \cite{MSemola}), Nobili and Violo \cite{NV-new} (see also \cite{NV,NV-adv}) proved recently  the following \textit{P\'olya--Szeg\H o inequality}:  for every $p>1$ and $u\in W^{ 1,p}_{\rm loc}(X,{\sf m})$, one has
\begin{equation}\label{Polya-Szego}
	{\sf Ch}_p(u)=\int_X |\nabla u|^p{\rm d} {\sf m}\geq {\sf AVR}(X)^\frac{p}{N}\int_0^\infty| (u^*)'|^p {\rm d}{\sf m}_N.
\end{equation}
Moreover, if the left hand side is finite then $u^*$ is locally absolutely continuous on its domain of definition. Inequality \eqref{Polya-Szego} is crucial in the proof of Theorem \ref{main-theorem-CD}. 
The analogue of \eqref{Polya-Szego} on Riemannian manifolds with non-negative Ricci curvature has been established in  \cite{BK}. 

The description of the equality case in the P\'olya--Szeg\H o inequality \eqref{Polya-Szego} is more delicate, even in the smooth setting. By using fine analysis of sets of finite perimeter,  Brothers and Ziemer \cite{BZ} characterized the equality in \eqref{Polya-Szego} in the Euclidean setting. Under certain regularity assumptions, a similar result to \cite{BZ} is provided  in \cite{BK} in the setting of Riemannian manifolds with non-negative Ricci curvature. As far as we know -- by using a general rearrangement technique  and  the sharp isoperimetric inequality \eqref{eqn-isoperimetric} -- the most general setting is handled by Nobili and Violo \cite[Theorem 1.3]{NV-new}: 
 if $(X,{\sf d}, {\sf m})$ is an essentially non-branching ${\sf CD}(0,N)$ space and equality holds for some  $u\in W^{ 1,p}_{\rm loc}(X,{\sf m})$ and 
both sides being non-zero and finite in \eqref{Polya-Szego}, then for some $x_0\in X$, the  space $(X,{\sf d}, {\sf m})$ is an $N$\textit{-volume cone} with the tip $x_0$, i.e., 
\begin{equation}\label{volume-cone}
	\frac{ {\sf m}(B_{x_0}(r))}{\omega_Nr^N}={\sf AVR}(X),\ \ \forall r>0.
\end{equation}
In addition, if $(u^*)'\neq 0$ a.e.\ on the set $\{{\rm ess\, inf}\, u<u^*<{\rm ess\, sup}\, u\}$ then $u$ is $x_0$-\textit{radial}, i.e., 
\begin{equation}\label{u=u^*}
	u(x)=u^*\left({\sf AVR}(X)^\frac{1}{N}{\sf d}(x_0,x)\right)\ \ {\sf m}{\rm -a.e.}\ x\in X.
\end{equation}
The result of Nobili and Violo \cite[Theorem 1.3]{NV-new} can be seen as a final product of several earlier works and ideas, arising from Antonelli,  Pasqualetto,  Pozzetta and  Violo \cite{APPV}, Antonelli,  Pasqualetto,  Pozzetta and  Semola \cite{APPS, APPS-2}, Balogh and Krist\'aly \cite{BK}, Cavalletti and   Manini \cite{CMan}, Mondino and Semola \cite{MSemola}, Nobili and Violo \cite{NV-adv} and  Pasqualetto and  Rajala \cite{Pasqualetto-Rajala}. 

\section{Proof of Theorem \ref{main-theorem-unified}}\label{section-unified}


We assume that the hypotheses (i)-(iii) of Theorem \ref{main-theorem-unified} are fulfilled.

Let  $x_0\in M$ from (iii), and according to (ii), we assume that \eqref{GN-metric-space} holds, i.e., there exists ${\sf C}>0$ such that
	\begin{equation}\label{GN-metric-space-biz}
	\|u\|_{L^q(X,{\sf m})}\leq {\sf C} \|\nabla u\|_{L^p(X,{\sf m})}^\theta \|u\|_{L^r(X,{\sf m})}^{1-\theta},\ \ u\in  W_{\rm loc}^{1,p}(X,{\sf m})\cap L^q(X,{\sf m}). 
\end{equation}

For every $R>0$, we consider the function 
$$
	u_R(x)=u_0\left(\frac{{\sf d}(x_0,x)}{R}\right),\ \ x\in X,
$$
where $u_0\in  W_{\rm loc}^{1,p}(\mathbb R_+,{\sf m}_N)\cap L^q(\mathbb R_+,{\sf m}_N)$ is a non-zero, non-negative and non-increasing extremizer in  \eqref{GN-1-dimensional-model} on the 1-dimensional model metric measure cone  $(\mathbb R_+,|\cdot|,{\sf m}_{N})$, i.e., 
\begin{equation}\label{egyenloseg-GN-ben-model-bizonyitas}
	\|u_0\|_{L^q(\mathbb R_+,{\sf m}_{N})}= {\sf K}_{\rm opt} \|u_0'\|_{L^p(\mathbb R_+,{\sf m}_{N})}^\theta \|u_0\|_{L^r(\mathbb R_+,{\sf m}_{N})}^{1-\theta},
\end{equation}
and  verifying \eqref{extremal-asymptotics}; in particular, all these integrals are finite and non-zero. If supp$(u_0)=[0,R_0]$ for some $R_0>0$  (with the convention that supp$(u_0)=\mathbb R_+$ when $R_0=\infty$), then  the support of $u_R$ is the closure of the ball $B_{x_0}(R_0 R)$, with the convention that   $B_{x_0}(R_0 R)=X$ whenever $R_0=\infty.$

We are going to use $u_R$ as a test function in \eqref{GN-metric-space-biz}. To do this, we recall from (iii) that the limit 
\begin{equation}\label{limit-L-infty}
	L_N(x_0)=\lim_{r\to \infty}\frac{	{\sf m}(B_{x_0}(r))}{\omega_N r^{N}}
\end{equation}
exists. Clearly, if $L_N(x_0)=+\infty$, we have nothing to prove; thus, we assume that ${\sf AVR}(X)=L_N(x_0)<+\infty$. We divide the proof into three steps. 

{\it \underline{Step 1: Lebesgue-norm estimates}.} Let us observe that $\rho\mapsto (u_0^q)'(\rho)\rho^N$ belongs to $L^1(0,R_0)$. Indeed, since $u_0\geq 0$ is a  non-increasing, non-zero extremal function in \eqref{GN-1-dimensional-model} (thus $u_0\in L^q(\mathbb R_+,{\sf m}_{N})$) such that  $u_0(0)$ is finite and $u_0^q(\rho)\rho^N\to 0$  as $\rho\to R_0$ (both for $R_0<\infty$ and $R_0=\infty$), see \eqref{extremal-asymptotics}, a partial integration yields that 
\begin{equation}\label{1-integralhatosag}
 0<-\int_0^{R_0} (u_0^q)'(\rho)\rho^N{\rm d}\rho =N\int_0^{R_0}  u_0^q(\rho)\rho^{N-1}{\rm d}\rho=\frac{1}{\omega_N}\int_0^{R_0}  u_0^q(\rho){\rm d}{\sf m}_N(\rho) =\frac{\|u_0\|^q_{L^q(\mathbb R_+,{\sf m}_{N})}}{\omega_N}<\infty.
\end{equation}
 The latter estimate also yields that   $\rho\mapsto (u_0^q)'(\rho){\sf m}\left(B_{x_0}(R\rho)\right)$ belongs to $L^1(0,R_0)$. Indeed, if $R_0<\infty$, the claim is trivial. When $R_0=\infty$, by  $L_N(x_0)<+\infty$, there exists $r_0>0$ such that
${\sf m}(B_{x_0}(r))\leq ({\sf AVR}(X)+1)\omega_N r^N$ for every $r>r_0$; therefore, 
\begin{align}\label{r-0-construction}
	\nonumber 0&<-\int_0^{\infty} (u_0^q)'(\rho){\sf m}(B_{x_0}(R\rho)){\rm d}\rho=-\left(\int_0^{r_0/R}+\int_{r_0/R}^\infty\right) (u_0^q)'(\rho){\sf m}(B_{x_0}(R\rho)){\rm d}\rho\\&\leq {\sf m}(B_{x_0}(r_0))(u_0^q(0)-u_0^q(r_0/R))-({\sf AVR}(X)+1)\omega_N R^N\int_{r_0/R}^\infty (u_0^q)'(\rho)\rho^N{\rm d}\rho<\infty.
\end{align}
 Having these facts, by Proposition \ref{proposition-co-area} one has that 
\begin{align}\label{u_metrikus-1}
	\nonumber	\|u_R\|^q_{L^q(X,{\sf m})}&=\int_{X}u_R^q{\rm d}{\sf m}=\int_{B_{x_0}(R_0R)}u_0^q\left(\frac{{\sf d}(x_0,x)}{R}\right){\rm d}{\sf m}(x)\\& \nonumber =
	\lim_{\rho\nearrow R_0} u_0^q(\rho){\sf m}(B_{x_0}(R\rho))  -\int_0^{R_0} (u_0^q)'(\rho)\,{\sf m}(B_{x_0}(R\rho)){\rm d}\rho
	\\&= -\int_0^{R_0} (u_0^q)'(\rho){\sf m}\left(B_{x_0}(R\rho)\right){\rm d}\rho,
\end{align}
where we used that $u_0(R_0)=0$ when $R_0<\infty$, and 
$\lim_{\rho\to \infty}u_0^q(\rho){\sf m}(B_{x_0}(R\rho)) =0$ when $R_0=\infty$, by using \eqref{limit-L-infty} and the assumption \eqref{extremal-asymptotics}. 
Since $\rho\mapsto (u_0^q)'(\rho)\rho^N$ belongs to $L^1(0,R_0)$, by the Lebesgue dominated convergence theorem and relations \eqref{u_metrikus-1}, \eqref{limit-L-infty} and \eqref{1-integralhatosag}  one has that 
\begin{align}\label{limit-q}
\nonumber	\lim_{R\to \infty}\frac{\|u_R\|^q_{L^q(X,{\sf m})}}{R^N}&=-\omega_N\lim_{R\to \infty}\int_0^{R_0} (u_0^q)'(\rho)\frac{ {\sf m}\left(B_{x_0}(R\rho)\right)}{\omega_N(R\rho)^N}\rho^N{\rm d}\rho=-{\sf AVR}(X)\omega_N\int_0^{R_0} (u_0^q)'(\rho)\rho^N{\rm d}\rho\\&={\sf AVR}(X)\|u_0\|^q_{L^q(\mathbb R_+,{\sf m}_{N})}.
\end{align} 

If $\theta<1$ (thus the last term in the Gagliardo--Nirenberg--Sobolev  inequality \eqref{GN-metric-space-biz}   does not vanish), a similar argument can be performed for the parameter $r>0$ instead of $q>0$, obtaining  that 
\begin{align}\label{limit-r}
	\lim_{R\to \infty}\frac{\|u_R\|^r_{L^r(X,{\sf m})}}{R^N}={\sf AVR}(X)\|u_0\|^r_{L^r(\mathbb R_+,{\sf m}_{N})}.
\end{align} 

{\it \underline{Step 2: Cheeger-energy estimate}.} Now, we shall focus on the 'gradient' term. First, since $u_R=u_0\circ ({\sf d}(x_0,\cdot)/R)$ is locally Lipschitz,  the minimal $p$-weak upper gradient of $u_R$ can be estimated by its local Lipschitz constant, i.e., $|\nabla u_R(x)|_p\leq |{\rm lip}_{\sf d}u_R|(x)$ for ${\sf m}$-a.e.\  $x\in X$, see \eqref{Cheeger-0-ineq}. Thus, by the  chain rule for locally Lipschitz functions,  we have for ${\sf m}$-a.e.\ $x\in X$ that 
$$|\nabla u_R(x)|_p\leq |{\rm lip}_{\sf d}u_R|(x)=\frac{1}{R}|u_0'|\left(\frac{{\sf d}(x_0,x)}{R}\right)|{\rm lip}_{\sf d} {\sf d}(x_0,\cdot)|(x).$$
Since ${\sf d}(x_0,\cdot)$ is 1-Lipschitz, it follows that $|{\rm lip}_{\sf d} {\sf d}(x_0,\cdot)|(x)\leq 1$ for ${\sf m}$-a.e.\ $x\in X$. 
Therefore, 
\begin{equation}\label{grad-estim-R}
	\|\nabla u_R\|^p_{L^p(X,{\sf m})}={\sf Ch}_p(u_R)=\int_X |\nabla u_R|^p_p{\rm d}{\sf m}\leq \frac{1}{R^p}\int_X |u_0'|^p\left(\frac{{\sf d}(x_0,x)}{R}\right){\rm d}{\sf m}.
\end{equation}

Note that 
$\rho\mapsto (|u_0'|^p)'(\rho)\rho^N\in L^1(0,R_0)$. If $R_0<\infty,$ the claim is trivial. If $R_0=\infty,$ since $u_0$ is non-increasing and  $u_0''(\rho)\geq 0$ for a.e.\ $\rho\in (i_0,\infty)$, see \eqref{extremal-asymptotics}, one has that $|(|u_0'|^p)'|=-(|u_0'|^p)'$ a.e.\ on $(i_0,\infty)$; thus, since $|u_0'|^p(\rho)\rho^N\to 0$ as $\rho\to \infty,$ see \eqref{extremal-asymptotics}, and $u_0'\in L^p(\mathbb R_+,{\sf m}_{N}),$ one has
\begin{align}\label{previous-estimate}
	\nonumber \left|\int_{\mathbb R_+} (|u_0'|^p)'(\rho)\rho^N{\rm d}\rho\right|&\leq \int_0^{i_0}  |(|u_0'|^p)'(\rho)|\rho^N{\rm d}\rho-\int_{i_0}^{\infty}(|u_0'|^p)'(\rho)\rho^N{\rm d}\rho\\&=\int_0^{i_0}  |(|u_0'|^p)'(\rho)|\rho^N{\rm d}\rho+|u_0'|^p(i_0)i_0^{N}+N\int_{i_0}^{\infty}|u_0'|^p(\rho)\rho^{N-1}{\rm d}\rho<\infty.
\end{align}
In particular, the latter integrability property and assumption \eqref{extremal-asymptotics} entitle  us to obtain  that
\begin{align}\label{u_R-1-model}
	\nonumber	\frac{\|u_0'\|_{L^p(\mathbb R_+,{\sf m}_{N})}^p}{\omega_N}&= N\int_{0}^{R_0}|u_0'|^p(\rho)\rho^{N-1}{\rm d}\rho\\&=\ds\left\{ 
	\begin{array}{lll}
		|u_0'|^p(R_0)R_0^N	-\ds\int_0^{R_0} (|u_0'|^p)'(\rho)\rho^N{\rm d}\rho,\ \ \text {if} \  R_0<\infty; \\
		-\ds\int_0^{\infty} (|u_0'|^p)'(\rho)\rho^N{\rm d}\rho,\ \ \ \ \ \ \ \ \ \ \ \ \ \ \ \ \ \ \ \  \text {if} \  R_0=\infty.
	\end{array}\right.
\end{align}

%

 We claim that  
$\rho\mapsto (|u_0'|^p)'(\rho){\sf m}(B_{x_0}(R\rho))$ belongs to $L^1(0,R_0)$. If $R_0<\infty,$ the claim is trivial as $|u_0'|^p\in BV([0,R_0])$, see \eqref{extremal-asymptotics}. If $R_0=\infty,$ let $l_0=\max\{i_0,r_0/R\}$ where $r_0>0$ is from \eqref{r-0-construction}. Then, since $|u_0'|^p\in BV([0,l_0])$,  the  estimate \eqref{previous-estimate} implies that
\begin{align}\label{derivative-oo00}
	\nonumber \int_{\mathbb R_+} \left|(|u_0'|^p)'(\rho)\right|{\sf m}(B_{x_0}(R\rho)){\rm d}\rho&\leq {\sf m}(B_{x_0}(R l_0))\int_0^{l_0}  |(|u_0'|^p)'(\rho)|{\rm d}\rho\\&\ \ \ \ -({\sf AVR}(X)+1)\omega_N R^N\int_{l_0}^\infty (u_0^p)'(\rho)\rho^N{\rm d}\rho<\infty,
\end{align}
which concludes the proof of the claim.

%
%

We are in the position to apply Proposition \ref{proposition-co-area} for $h:=|u_0'|^p(\cdot/R)$ on $[0,RR_0),$ obtaining through \eqref{grad-estim-R}, \eqref{extremal-asymptotics}  and a change of variables that 
\begin{align*}\label{u_R-grad-estimate}
\nonumber	\|\nabla u_R\|^p_{L^p(X,{\sf m})}&\leq \frac{1}{R^p}\int_{B_{x_0}(R_0R)} |u_0'|^p\left(\frac{{\sf d}(x_0,x)}{R}\right){\rm d}{\sf m}\\&=\frac{1}{R^p}\ds\left\{ 
\begin{array}{lll}
|u_0'|^p(R_0){\sf m}\left(B_{x_0}(RR_0)\right)	-\ds\int_0^{R_0} (|u_0'|^p)'(\rho){\sf m}\left(B_{x_0}(R\rho)\right){\rm d}\rho,\ \ \text {if} \  R_0<\infty; \\
-\ds\int_0^{\infty} (|u_0'|^p)'(\rho){\sf m}\left(B_{x_0}(R\rho)\right){\rm d}\rho,\ \ \ \ \ \ \ \ \ \ \ \ \ \ \ \ \ \ \ \ \ \ \ \ \ \ \ \ \ \ \ \ \ \ \text {if} \  R_0=\infty.
\end{array}\right.
\end{align*}
Combining the latter estimate with the Lebesgue dominated convergence theorem and relations \eqref{limit-L-infty} and \eqref{u_R-1-model}, it follows that 
 \begin{equation}\label{limit-p}
 		\lim_{R\to \infty}\frac{\|\nabla u_R\|^p_{L^p(X,{\sf m})}}{R^{N-p}}\leq {\sf AVR}(X)\|u_0'\|^p_{L^p(\mathbb R_+,{\sf m}_{N})}.
 \end{equation} 

{\it \underline{Step 3: Blow-down limiting argument}.}
By \eqref{u_metrikus-1}, it follows that $u_R\in L^q(X,{\sf m})$ for every $R>0$. Moreover, $u_R\in  W_{\rm loc}^{1,p}(X,{\sf m})$ since for every $\eta\in {\rm Lip}_{bs}(X)$ one has that $\eta u_R\in W^{1,p}(X,{\sf m})$. Indeed, if $K\subset X$ is the bounded support of $\eta$ and $m_K:=\max_{\overline K}|\eta|$, we first have that
$$\|\eta u_R\|_{L^p(X,{\sf m})}^p=\int_K|\eta|^pu_R^p{\rm d}{\sf m}\leq m_K^pu_0^p(0){\sf m}(K)<\infty.$$
Moreover, by Gigli \cite[relation (2.13)]{Gigli}, one has ${\sf m}$-a.e.\ that $|\nabla(\eta u_R)|_p\leq |\eta||\nabla  u_R|_p+u_R|\nabla  \eta|_p\leq |\eta||\nabla  u_R|_p+L_Ku_R,$ where $L_K>0$ is the Lispchitz constant of $\eta$ on $K$; 
thus 
$$\|\nabla (\eta u_R)\|^p_{L^p(X,{\sf m})}\leq 2^p(m_K^p\|\nabla  u_R\|^p_{L^p(X,{\sf m})}+L^pu_0^p(0){\sf m}(K))<\infty.$$


Accordingly,  $u_R\in  W_{\rm loc}^{1,p}(X,{\sf m})\cap L^q(X,{\sf m})$ for every $R>0$, which can be used in \eqref{GN-metric-space-biz}
 as a test function, i.e., one has 
 $$\|u_R\|_{L^q(X,{\sf m})}\leq {\sf C} \|\nabla u_R\|_{L^p(X,{\sf m})}^\theta \|u_R\|_{L^r(X,{\sf m})}^{1-\theta}.$$ Let us divide this inequality by $R^{N/q}$ and taking into account that $\frac{N}{q}=\frac{N-p}{p}\theta+\frac{N}{r}(1-\theta)$, see the balance condition \eqref{balance-condition}, we obtain  that
 $$\frac{\|u_R\|_{L^q(X,{\sf m})}}{R^\frac{N}{q}}\leq {\sf C} \left(\frac{\|\nabla u_R\|_{L^p(X,{\sf m})}}{R^\frac{N-p}{p}}\right)^\theta \left(\frac{\|u_R\|_{L^r(X,{\sf m})}}{R^\frac{N}{r}}\right)^{1-\theta}.$$
Letting  $R\to \infty$ in the last inequality, by relations \eqref{limit-q}, \eqref{limit-r} and  \eqref{limit-p} it follows that 
$$\left({\sf AVR}(X)\|u_0\|^q_{L^q(\mathbb R_+,{\sf m}_{N})}\right)^\frac{1}{q}\leq {\sf C} \left({\sf AVR}(X)\|u_0'\|^p_{L^p(\mathbb R_+,{\sf m}_{N})}\right)^\frac{\theta}{p} \left({\sf AVR}(X)\|u_0\|^r_{L^r(\mathbb R_+,{\sf m}_{N})}\right)^\frac{{1-\theta}}{r}.$$
The equality \eqref{egyenloseg-GN-ben-model-bizonyitas} and a reorganization  of the terms in the last inequality imply that 
$${\sf K}_{\rm opt}\leq {\sf C}\cdot {\sf AVR}(X)^{\frac{\theta}{p}+\frac{{1-\theta}}{r}-\frac{1}{q}}.$$
Since $\frac{\theta}{p}+\frac{{1-\theta}}{r}-\frac{1}{q}=\frac{\theta}{N}$, the latter inequality is precisely the required relation  \eqref{volume-growth-estimate}. 
\hfill$\square$

\begin{remark}\label{remark-b}\rm (a) Beside (iii) in Theorem \ref{main-theorem-unified},  if we assume that for some $x_0\in X$ the  local density
	$L_0(x_0)=\lim_{r\to 0}\frac{	{\sf m}(B_{x_0}(r))}{\omega_N r^{N}}$
exists, a similar argument as in the above proof shows that  
\begin{equation}\label{volume-growth-estimate-in-zero}
	L_0(x_0)\geq \left(\frac{{\sf K}_{\rm opt}}{\sf C}\right)^\frac{N}{\theta}.
\end{equation}
In particular, if $(X,{\sf d},{\sf m})=(M,{\sf d}_g,{\rm d}v_g)$ is an $n$-dimensional Riemannian manifold, then for every $x_0\in X$ one has that $L_0(x_0)=1$, see Gallot,  Hulin and Lafontaine \cite[Theorem 3.98]{GHL}. Therefore, by \eqref{volume-growth-estimate-in-zero} we obtain that ${\sf C}\geq {\sf K}_{\rm opt}$. Such conclusions have been obtained by using local charts, see Aubin \cite{Aubin} and Hebey \cite[Chapter 4]{Hebey}. 

(b) For simplicity, let $\theta=1$ in Theorem \ref{main-theorem-unified} (thus $q=Np/(N-p)$). If the limit \eqref{Limit-exists-N} in Theorem \ref{main-theorem-unified} does not exist, we consider
	$$l=\liminf_{r\to \infty}\frac{	{\sf m}(B_{x_0}(r))}{\omega_N r^{N}}\ \ {\rm and}\ \ L=\limsup_{r\to \infty}\frac{	{\sf m}(B_{x_0}(r))}{\omega_N r^{N}}.
	$$
Fatou's lemma and a similar blow-down argument as in the  proof Theorem \ref{main-theorem-unified} imply
	that
			\begin{equation}\label{L-l}
					L\left(\frac{L}{l}\right)^\frac{N(p-1)+1}{q(p-1)}\geq  \left(\frac{{\sf K}_{\rm opt}}{\sf C}\right)^N.
				\end{equation}
			As expected,   \eqref{L-l} reduces  to \eqref{volume-growth-estimate} whenever $L=l$. 

\end{remark}

%


\section{Proof of Theorem \ref{main-theorem-CD}} \label{section-4}

We divide the proof into three parts. 

\textit{\underline{Step 1: proof of \eqref{GN-CD}}.} Let us fix $u\in W_{\rm loc}^{1,p}(X,{\sf m})\cap L^q(X,{\sf m})$; in order to prove \eqref{GN-CD}, without loss of generality, we may assume that $u$ is non-negative. Let $u^*$ be the non-increasing rearrangement of $u$, defined in \S \ref{section-rearrangement}; in particular, $u^*$ verifies the Gagliardo--Nirenberg--Sobolev  inequality \eqref{GN-1-dimensional-model}. Therefore, combining the Cavalieri principle \eqref{Cavallieri} and P\'olya--Szeg\H o inequality   \eqref{Polya-Szego}, it follows  by \eqref{GN-1-dimensional-model} that 
\begin{align}\label{estimate-GN-final}
	\nonumber 	\|u\|_{L^q(X,{\sf m})}&=\|u^*\|_{L^q(\mathbb R_+,{ \sf m}_N)}\\& \leq  {\sf K}_{\rm opt} \| (u^*)'\|_{L^p(\mathbb R_+,{ \sf m}_N)}^{\theta} \|u^*\|_{L^r(\mathbb R_+,{ \sf m}_N)}^{1-\theta}\\& \label{estimate-GN-2-final}
	\leq {\sf K}_{\rm opt}{\sf AVR}(X)^{-\frac{\theta}{N}} \|\nabla u\|_{L^p(X,{\sf m})}^\theta \|u\|_{L^r(X,{\sf m})}^{1-\theta},
\end{align}
which is precisely inequality  \eqref{GN-CD}. 

\textit{\underline{Step 2: sharpness of \eqref{GN-CD}}.} We assume by contradiction that there exists ${\sf C}>0$ such that ${\sf C}<{\sf K}_{\rm opt}{\sf AVR}(X)^{-\frac{\theta}{N}}$ and 
\begin{equation}\label{GN-CD-modified}
	\|u\|_{L^q(X,{\sf m})}\leq {\sf C}\|\nabla u\|_{L^p(X,{\sf m})}^\theta \|u\|_{L^r(X,{\sf m})}^{1-\theta},\ \ \forall u\in W_{\rm loc}^{1,p}(X,{\sf m})\cap L^q(X,{\sf m}). 
\end{equation}
Since $(X,{\sf d},{\sf m})$ is a  ${\sf CD}(0,N)$ metric measure space,  the limit $L_N(x_0)$ in \eqref{Limit-exists-N} exists and 
\begin{equation}\label{L_N-CD-letezik}
	L_N(x_0)={\sf AVR}(X),\ \ \forall x_0\in X.
\end{equation}
Therefore, by the latter relation and Theorem \ref{main-theorem-unified},  for every $x_0\in X$ we have
$$ {\sf AVR}(X)=	L_N(x_0)\geq \left(\frac{{\sf K}_{\rm opt}}{\sf C}\right)^\frac{N}{\theta},$$
which contradicts our assumption ${\sf C}<{\sf K}_{\rm opt}{\sf AVR}(X)^{-\frac{\theta}{N}}$. 

\textit{\underline{Step 3: equality in \eqref{GN-CD}}.}

(i)$\implies$(ii). Assume that equality holds in \eqref{GN-CD} for some non-zero  $u\in W_{\rm loc}^{1,p}(X,{\sf m})\cap L^q(X,{\sf m})$. In particular, we have equalities in \eqref{estimate-GN-final} and \eqref{estimate-GN-2-final}, respectively. On the one hand, the equality in \eqref{estimate-GN-final} implies that $u^*$ is an extremizer in the Gagliardo--Nirenberg--Sobolev  inequality \eqref{GN-1-dimensional-model}. Since -- by assumption -- the extremal function $u_0$ in  \eqref{GN-1-dimensional-model}  is unique $($up to scaling and multiplicative factor$)$, we have that $u^*=u_0.$  On the other hand, since $(X,{\sf d}, {\sf m})$ is an essentially non-branching ${\sf CD}(0,N)$ space, the equality in \eqref{estimate-GN-2-final} implies the equality in the P\'olya--Szeg\H o inequality \eqref{Polya-Szego}, thus $X$ is an $N$-volume cone with tip $x_0$ for some $x_0\in X$, see \eqref{volume-cone}. Moreover, since by assumption $(u^*)'\neq 0$ a.e.\ on the set $\{{\rm ess\, inf}\, u<u^*<{\rm ess\, sup}\, u\},$ then $u$ is $x_0$-radial, see \eqref{u=u^*}, i.e., $	u(x)=u^*\left({\sf AVR}(X)^\frac{1}{N}{\sf d}(x_0,x)\right)$ for $ {\sf m}{\rm -a.e.}\ x\in X.$ Since  {\rm supp}$(u_0)=[0,R_0]$ $($with  {\rm supp}$(u_0)=\mathbb R_+$ whenever $R_0=\infty$$)$, the latter relation is valid in fact for  $ {\sf m}{\rm -a.e.}\ x\in B_{x_0}(R_0/{\sf AVR}(X)^{\frac{1}{N}})$.

(ii)$\implies$(i). We assume that $X$ is an $N$-volume cone with a tip $x_0\in X$ and up to scaling and multiplicative factor,  $$u(x)=u_0\left({\sf AVR}(X)^{\frac{1}{N}}{\sf d}(x_0,x)\right),\ \ x\in B_{x_0}(R_0/{\sf AVR}(X)^{\frac{1}{N}}),$$ where $u_0\geq 0$ is an extremal $C^1$-function  in \eqref{GN-1-dimensional-model} verifying \eqref{extremal-asymptotics}, and  {\rm supp}$(u_0)=[0,R_0]$ $($with  {\rm supp}$(u_0)=\mathbb R_+$ whenever $R_0=\infty$$)$. 
By Proposition \ref{proposition-co-area} and the fact that 
${\sf m}(B_{x_0}(r))={\sf AVR}(X){\omega_Nr^N}$ for all $r>0,$ see \eqref{volume-cone}, 
a similar computation as in \eqref{u_metrikus-1} shows that 
\begin{align}\label{u_metrikus-2}
	\nonumber	\|u\|^q_{L^q(X,{\sf m})}&=\int_{B_{x_0}(R_0/{\sf AVR}(X)^{\frac{1}{N}})}u_0^q\left({\sf AVR}(X)^{\frac{1}{N}}{\sf d}(x_0,x)\right){\rm d}{\sf m}(x)\\& \nonumber =
	\lim_{\rho\nearrow R_0} u_0^q(\rho){\sf m}(B_{x_0}({\sf AVR}(X)^{-\frac{1}{N}}\rho))  -\int_0^{R_0} (u_0^q)'(\rho)\,{\sf m}(B_{x_0}({\sf AVR}(X)^{-\frac{1}{N}}\rho)){\rm d}\rho
	\\&\nonumber = -\int_0^{R_0} (u_0^q)'(\rho){\sf m}\left(B_{x_0}({\sf AVR}(X)^{-\frac{1}{N}}\rho)\right){\rm d}\rho= -\omega_N\int_0^{R_0} (u_0^q)'(\rho)\rho^N{\rm d}\rho\\& = N\omega_N\int_0^{R_0} u_0^q(\rho)\rho^{N-1}{\rm d}\rho= \|u_0\|^q_{L^q(\mathbb R_+,{\sf m}_{N})}.
\end{align}
If $\theta<1$, in a similar manner as before, we also have that 
\begin{equation}\label{r-re-vonatkozo-egyenloseg}
	\|u\|_{L^r(X,{\sf m})}=\|u_0\|_{L^r(\mathbb R_+,{\sf m}_{N})}.
\end{equation}

By  \eqref{Cheeger-equality} and the eikonal equation \eqref{eikonal}, the  chain rule for locally Lipschitz functions provides 
$$|\nabla u |=|{\rm lip}_{\sf d}u|=-{\sf AVR}(X)^{\frac{1}{N}}u_0'\left({\sf AVR}(X)^{\frac{1}{N}}{\sf d}(x_0,\cdot )\right)\ \ {\sf m}{\rm -a.e.\ on}\  X.$$
Therefore, by Proposition \ref{proposition-co-area} and ${\sf m}(B_{x_0}(r))={\sf AVR}(X){\omega_Nr^N}$ for all $r>0,$ it yields that
\begin{align}\label{gradiens-N-cone}
	\nonumber	\|\nabla u\|^p_{L^p(X,{\sf m})}&=\int_X |\nabla u|^p{\rm d}{\sf m}={\sf AVR}(X)^{\frac{p}{N}}\int_{B_{x_0}(R_0/{\sf AVR}(X)^{\frac{1}{N}})} |u_0'|^p\left({\sf AVR}(X)^{\frac{1}{N}}{\sf d}(x_0,x )\right){\rm d}{\sf m}(x)\\&\nonumber=\omega_N{\sf AVR}(X)^{\frac{p}{N}}\ds\left\{ 
	\begin{array}{lll}
		|u_0'|^p(R_0)R_0^N	-\ds\int_0^{R_0} (|u_0'|^p)'(\rho)\rho^N{\rm d}\rho,\ \ \text {if} \  R_0<\infty; \\
		-\ds\int_0^{\infty} (|u_0'|^p)'(\rho)\rho^N{\rm d}\rho,\ \ \ \ \ \ \ \ \ \   \ \ \ \ \ \ \ \ \ \  \text {if} \  R_0=\infty;
	\end{array}\right.
	\\&=N\omega_N{\sf AVR}(X)^{\frac{p}{N}}\ds \int_0^{R_0} |u_0'|^p(\rho)\rho^{N-1}{\rm d}\rho={\sf AVR}(X)^{\frac{p}{N}}\|u_0'\|^p_{L^p(\mathbb R_+,{\sf m}_{N})}.
\end{align}
Since $u_0$ is an extremizer in the Gagliardo--Nirenberg--Sobolev inequality \eqref{GN-1-dimensional-model} with the optimal constant ${\sf K}_{\rm opt}$, relations \eqref{u_metrikus-2}, \eqref{r-re-vonatkozo-egyenloseg} and \eqref{gradiens-N-cone} imply that $u$ verifies the equality in \eqref{GN-CD}. \hfill $\square$

\begin{remark}\rm \label{remark-particularize} (a)
	Theorem \ref{main-theorem-CD} -- and its particular/borderline forms in the next subsections can be particularized to Riemannian and Finsler manifolds with non-negative Ricci curvature, see   Sturm \cite{Sturm-2} for Riemannian manifolds and Ohta \cite{Ohta} for Finsler manifolds, respectively. 
	
	(b)	If we particularize Theorem \ref{main-theorem-CD} to ${\sf RCD}(0,N)$ spaces, i.e., ${\sf CD}(0,N)$ spaces with infinitesimally Hilbertian structure, see Gigli \cite{Gigli}, it turns out that $N$-volume cones become \textit{metric} cones, see  De Philippis and Gigli \cite[Theorem 3.39]{DeP-Gigli}, i.e., \eqref{volume-cone} implies that there exists an
	${\sf RCD}(N-2, N-1)$ space $(Z, {\sf d}_Z, {\sf m}_Z)$ with diam$(Z) \leq \pi$ such that the ball $B_{x_0}(R)\subset X$
	is isometric to the ball $B_{O_Y}(R)$ of the cone $Y$ built over $Z$ for every $R>0.$ Clearly, $n$-dimensional  Riemannian manifolds with non-negative Ricci curvature fall into this class, and subsequently, $"X$ is an $n$-volume cone" is understood as $"X$ is isometric to the Euclidean space $\mathbb R^n$". In the class of Finsler manifolds, similar characterization is not available; see e.g.\  Shen \cite{Shen-1, Shen-2}.
\end{remark}

\section{Particular cases of Gagliardo--Nirenberg--Sobolev inequalities} \label{section-5}
In this section we show the applicability of Theorems \ref{main-theorem-unified} and \ref{main-theorem-CD};  in \S\ref{subsection-GN} we focus on  inequalities having Barenblatt profiles as extremizers in the $1$-dimensional model cone  $(\mathbb R_+,|\cdot|,{\sf m}_{N}),$ while 
 in \S\ref{subsection-Nash}, we deal with inequalities having non-Barenblatt profiles.  

 
\subsection{Inequalities with Barenblatt-type extremizers}\label{subsection-GN} Several inequalities can be included into this class; we discuss in detail a Gag\-liardo--Ni\-ren\-berg--Sobolev interpolation inequality, while for the others we only indicated the necessary ingredients. 

\subsubsection{Gag\-liardo--Ni\-ren\-berg--Sobolev interpolation inequality I} 
Let $N>1$, $p\in (1,N)$, $p^\star=\frac{pN}{N-p}$, $p'=\frac{p}{p-1}$ the conjugate of $p$,   $\alpha\in (1,\frac{N}{N-p}]$ and consider the numbers 
\begin{equation}\label{theta-best}
	\theta=\frac{p^\star(\alpha-1)}{\alpha p(p^\star-\alpha
		p+\alpha-1)},
\end{equation}
and  
\begin{equation}\label{GN-best-constant-1}
	\mathcal
	G_{\alpha,p,N}=\left(\frac{\alpha-1}{p'}\right)^\theta
	\frac{\left(\frac{p'}{N}\right)^{\frac{\theta}{p}+\frac{\theta}{N}}\left(\frac{\alpha (p-1)+1}{\alpha
			-1}-\frac{N}{p'}\right)^\frac{1}{\alpha p}
		\left(\frac{\alpha (p-1)+1}{\alpha
			-1}\right)^{\frac{\theta}{p}-\frac{1}{\alpha p}}}{\left(\omega_N {\sf B}\left(\frac{\alpha (p-1)+1}{\alpha
			-1}-\frac{N}{p'},\frac{N}{p'}\right)\right)^{\frac{\theta}{N}}},
\end{equation}	where ${\sf B}$ stands for the Beta-function. One can easily show that $\theta\in (0,1]$. 

\begin{theorem}\label{GN-1-theorem}   Let $N>1$, $p\in (1,N)$ and  $\alpha\in (1,\frac{N}{N-p}]$, let  $(X,{\sf d},{\sf m})$ be a metric measure space 	
	supporting the  Gagliardo--Nirenberg--Sobolev  inequality  
	\begin{equation}\label{GN-1-metric-space}
		\|u\|_{L^{\alpha p}(X,{\sf m})}\leq {\sf C}
		\|\nabla u\|_{L^p(X,{\sf m})}^{\theta}\|u\|_{L^{\alpha(p-1)+1}(X,{\sf m})}^{1-\theta},\
		\forall  u\in  W_{\rm loc}^{1,p}(X,{\sf m})\cap L^{\alpha p}(X,{\sf m}),
	\end{equation}
	for some ${\sf C}>0$	with $\theta\in (0,1]$ from \eqref{theta-best}, 
	and assume that for some $x_0\in X$ the  limit $L_N(x_0)$
	exists in \eqref{Limit-exists-N}. 
	Then, 
	\begin{equation}\label{GN-2-volume-growth-estimate}
		{\sf AVR}(X)\geq \left(\frac{\mathcal
			G_{\alpha,p,N}}{{\sf C}}\right)^\frac{N}{\theta}.
	\end{equation}
\end{theorem}

\begin{proof}
	In Theorem \ref{main-theorem-unified}, we choose $q=\alpha p$ and $r=\alpha(p-1)+1$. With these choices and $\theta$ from \eqref{theta-best} we have the balance condition \eqref{balance-condition}. Moreover, according to Lam \cite{Lam},  Balogh, Don and Krist\'aly \cite[Theorem 3.1/(i)]{BDK} or Cordero-Erausquin, Nazaret and Villani \cite{CENV} and  Del Pino and Dolbeault \cite{delPino-JMPA} (for $N=n\in \mathbb N$), one has the sharp weighted Gagliardo--Nirenberg--Sobolev inequality 
	\begin{equation}\label{gn-elso}
		\|u\|_{L^{\alpha p}(\mathbb R_+,{ \sf m}_N)}\leq\mathcal G_{\alpha,p,N}
		\|u'\|_{L^p(\mathbb R_+,{ \sf m}_N)}^{\theta}\|u\|_{L^{\alpha(p-1)+1}(\mathbb R_+,{ \sf m}_N)}^{1-\theta},\
		\forall  u\in  W_{\rm loc}^{1,p}(\mathbb R_+,{ \sf m}_N)\cap L^{\alpha p}(\mathbb R_+,{ \sf m}_N),
	\end{equation}
	where ${\sf m}_N=N\omega_Nr^{N-1}\mathcal L^1$, while the unique extremal is -- up to scaling and multiplicative factor --  the Barenblatt function $$u_0(t)=(1+t^{p'})^\frac{1}{1-\alpha},\ t\geq 0.$$
	In particular, supp$(u_0)=[0,\infty)$, thus we may choose $R_0=\infty$ in Theorem \ref{main-theorem-unified}. Since $p\in (1,N)$ and  $\alpha\in (1,\frac{N}{N-p}]$, it is easy to verify the assumptions from \eqref{extremal-asymptotics} by choosing $i_0=\left(\frac{\alpha-1}{\alpha(p-1)+1}\right)^{1/p'}$.	 It remains to apply  Theorem \ref{main-theorem-unified}.
\end{proof}

\begin{remark}\rm \label{remark-Sobolev} (a) Theorem \ref{GN-1-theorem} is well-known on ${\sf CD}(0,N)$ spaces by Krist\'aly \cite{Kristaly-Calculus-2016},  and on Riemannian manifolds with non-negative Ricci curvature by Xia \cite{Xia-JFA}; these proofs used the explicit Barenblatt profile of the extremizer $u_0.$

	(b)
	When $\theta=1$, which is equivalent to $\alpha=\frac{N}{N-p}$, the Gagliardo--Nirenberg--Sobolev inequality \eqref{GN-1-metric-space} reduces to the well-known Sobolev inequality: 
	\begin{equation}\label{Sob-1-metric-space}
		\|u\|_{L^{p^\star}(X,{\sf m})}\leq {\sf C}
		\|\nabla u\|_{L^p(X,{\sf m})},\
		\forall  u\in  W_{\rm loc}^{1,p}(X,{\sf m})\cap L^{ p^\star}(X,{\sf m}).
	\end{equation}
	Thus, Theorem \ref{main-theorem-unified} applies, obtaining that whenever \eqref{Sob-1-metric-space} holds and  $L_N(x_0)$
	exists in \eqref{Limit-exists-N},  then $${\sf AVR}(X)\geq \left(\frac{AT(p,N)}{{\sf C}}\right)^{N},$$
	where $AT(p,N)=\mathcal
	G_{\frac{N}{N-p},p,N}$ is the optimal Aubin--Talenti constant. 
	In the particular case, when $(X,{\sf d},{\sf m})=(M,{\sf d}_g,{\rm d}v_g)$ is an $n$-dimensional Riemannian manifold, the latter result  has been established by Carron \cite{Carron}. 
\end{remark}

A direct consequence of Theorem \ref{main-theorem-CD}, combined with the  proof of Theorem \ref{GN-1-theorem}, yields: 

	\begin{theorem}\label{GN-1-theorem-CD}   Let $N>1$, $p\in (1,N)$ and  $\alpha\in (1,\frac{N}{N-p}]$,  and $(X,{\sf d},{\sf m})$ be an essentially non-branching  ${\sf CD}(0,N)$ metric measure space with ${\sf AVR}(X) >0.$ Then one has the following  Gagliardo--Nirenberg--Sobolev  inequality  
	\begin{equation}\label{GN-1-metric-space-CD}
		\|u\|_{L^{\alpha p}(X,{\sf m})}\leq {\sf AVR}(X)^{-\frac{\theta}{N}}\mathcal
		G_{\alpha,p,N}
		\|\nabla u\|_{L^p(X,{\sf m})}^{\theta}\|u\|_{L^{\alpha(p-1)+1}(X,{\sf m})}^{1-\theta},\
		\forall  u\in  W_{\rm loc}^{1,p}(X,{\sf m})\cap L^{\alpha p}(X,{\sf m}),
	\end{equation}
	and ${\sf AVR}(X)^{-\frac{\theta}{N}}\mathcal
	G_{\alpha,p,N}$ is sharp.
	
	Moreover, equality holds in \eqref{GN-1-metric-space-CD} for some non-zero  $u\in  W_{\rm loc}^{1,p}(X,{\sf m})\cap L^{\alpha p}(X,{\sf m})$ 
	if and only if 
	$X$ is an $N$-volume cone with a tip $x_0\in X$ and up to scaling and multiplicative factor,  $u(x)=(1+{\sf d}^{p'}(x_0,x))^\frac{1}{1-\alpha}$ for ${\sf m}$-a.e.\ $x\in X$. 
\end{theorem}


\subsubsection{Gag\-liardo--Ni\-ren\-berg--Sobolev interpolation inequality II} \label{subsubsec-5-1-2}
Let us fix $N>1$, $p\in (1,N)$ and 
	$0<\alpha<1$. Let $q=\alpha(p-1)+1$, $r=\alpha p$ and $\theta=\gamma=\frac{p^\star(1-\alpha)}{(p^\star-\alpha
				p)(\alpha p+1-\alpha)}$; it is clear that these numbers verify the balance condition \eqref{balance-condition}.  Due  to Lam \cite{Lam},  Balogh, Don and Krist\'aly \cite[Theorem 3.1/(ii)]{BDK} or Cordero-Erausquin, Nazaret and Villani \cite{CENV} (for $N=n\in \mathbb N$), the following sharp weighted Gagliardo--Nirenberg--Sobolev inequality holds: 
	\begin{equation}\label{gn-masodik}
		\|u\|_{L^{\alpha(p-1)+1}(\mathbb R_+,{ \sf m}_N)}\leq\mathcal N_{\alpha,p,N}
		\|u'\|_{L^p(\mathbb R_+,{ \sf m}_N)}^{\gamma}\|u\|_{L^{\alpha p}(\mathbb R_+,{ \sf m}_N)}^{1-\gamma},\
		\forall  u\in  W_{\rm loc}^{1,p}(\mathbb R_+,{ \sf m}_N)\cap L^{\alpha(p-1)+1}(\mathbb R_+,{ \sf m}_N),
	\end{equation}
	where the optimal constant $\mathcal N_{\alpha,p,N}$ can be obtained by replacing the Barenblatt function   
	$u_0(t)=(1-t^{p'})_+^\frac{1}{1-\alpha},\ t\geq 0,$ into \eqref{gn-masodik}. Note that $u_0$ is an  
	extremal of \eqref{gn-masodik},  supp$(u_0)=[0,1]$ and we may choose $R_0=1$; moreover, $u_0\geq 0$ and  $u_0'\leq 0$ on $[0,1]$. It remains to apply  Theorems \ref{main-theorem-unified} and \ref{main-theorem-CD} in order to state similar results as Theorems \ref{GN-1-theorem} and \ref{GN-1-theorem-CD}, respectively.

\subsubsection{Faber--Krahn  inequality I}
If $\alpha\to 0$ in the Gagliardo--Nirenberg--Sobolev inequality \eqref{gn-masodik}, we obtain the $L^1$-Faber--Krahn inequality in the 1-dimensional model cone $(\mathbb R_+,|\cdot|,{ \sf m}_N)$:
\begin{equation}\label{Faber-Krahn-1-dim-elso}
	\|u\|_{L^1(\mathbb R_+,{ \sf m}_N)}\leq \mathcal F_{p,N}\|u'  \|_{L^{p}(\mathbb R_+,{ \sf m}_N)}{\sf m}_N({\rm supp}(u))^{1-\frac{1}{p^\star}}, \ \forall u\in  W_{\rm loc}^{1,p}(\mathbb R_+,{ \sf m}_N)\cap L^1(\mathbb R_+,{ \sf m}_N),
\end{equation}
where $\mathcal F_{p,N}=N^{-\frac{1}{p}}\omega_N^{-\frac{1}{N}}(p'+N)^{-\frac{1}{p'}}$ can be obtained by using the Barenblatt function  
$	u_0(t)=(1-t^{p'})_+,\ t\geq 0,
$
which is an extremizer in \eqref{Faber-Krahn-1-dim-elso}, see 
Balogh, Don and Krist\'aly \cite[Theorem 3.3]{BDK} or Cordero-Erausquin, Nazaret and Villani \cite[p.\ 320]{CENV} (for $N=n\in \mathbb N$). Here, supp$(u)$ stands for the support of $u$, and a direct calculations show that the latter term in \eqref{gn-masodik} reduces to 
\begin{align*}
	\lim_{\alpha\to 0}\|u\|_{L^{\alpha p}(\mathbb R_+,{ \sf m}_N)}^{1-\gamma}
	={\sf m}_N({\rm supp}(u))^{1-\frac{1}{p^\star}},
\end{align*}
where $\gamma$ is from \S \ref{subsubsec-5-1-2}. Now, we can prove similar results to Theorems \ref{GN-1-theorem} and \ref{GN-1-theorem-CD}, respectively.
%
%

\subsubsection{Morrey--Sobolev inequality} Let $p > N >1$ and let us choose $q=\infty$ and $r\to 0$ in \eqref{GN-1-dimensional-model}, which reduces to the $L^\infty$-Morrey--Sobolev inequality 
\begin{equation}\label{Talenti-2}
	\| u\|_{L^\infty(\mathbb R_+)} \leq\,{\sf K}_{\rm opt} \|  u'\|_{L^p(\mathbb R_+,{\sf m}_N)}{\sf m}_N({\rm supp}\, u)^{\frac{1}{N}-\frac{1}{p}}  ,\ \ \ \forall u\in  W_{\rm loc}^{1,p}(\mathbb R_+,{\sf m}_N)\cap L^\infty(\mathbb R_+).
\end{equation}
Indeed, the balance condition \eqref{balance-condition} reduces to 
$0=\frac{\theta}{p}-\frac{\theta}{N}+\frac{1-\theta}{r}$ with $r\to 0$, thus $\theta\to 1$ and 
\begin{align*}
	\lim_{r\to 0}\|u\|_{L^{r}(\mathbb R_+,{ \sf m}_N)}^{1-\theta}={\sf m}_N({\rm supp}(u))^{\frac{1}{N}-\frac{1}{p}} .
\end{align*}
It turns out that $	{\sf K}_{\rm opt}={\sf
	T}_{p,N}=N^{-\frac{1}{p}}\omega_N^{-\frac{1}{N}}\left(\frac{p-1}{p-N}\right)^\frac{1}{p'},$ which  can be obtained by the  Barenblatt-type function 
$
u_0(t)=
\left(1-t^\frac{p-N}{p-1}\right)_+,\ t\geq 0,
$
which is the (unique) extremal in \eqref{Talenti-2}, see 
Talenti \cite[Theorem 2.E]{Talenti} (where the case $N\in \mathbb N$ is considered, which can be easily extended to generic $N>1$). The rest is similar as before. 

\subsection{Inequalities with non-Barenblatt-type extremizers} In this subsection we discuss in detail the Nash inequality, and then we sketch the ingredients for another Faber--Krahn inequality (concerning the classical Dirichlet $p$-eigenvalue problem). 

\subsubsection{Nash inequality}\label{subsection-Nash}
We start with an auxiliary result for Bessel functions, which will be important in the sequel. To do this, we recall that $j_\nu:=j_{\nu,1}$ stands for the first positive root of the Bessel function $J_\nu$ of the first kind and order $\nu\geq 0$. 
\begin{proposition}\label{Bessel-lemma}
	Let $\nu\geq 0$. Then 
	\begin{equation}\label{Bessel-inequality}
		\frac{J_\nu(j_{\nu+1}t)}{J_{\nu}(j_{\nu+1})}\leq t^\nu,\ \ \forall t\in [0,1]. 
	\end{equation}
\end{proposition}

\begin{proof}
	Note first that $J_{\nu}(j_{\nu+1})<0$. Indeed, 
	the zeros of Bessel functions interlace according to the inequality $j_\nu<j_{\nu+1}<j_{\nu,2}<j_{\nu+1,2}<...$, where $j_{\nu,k}$ denotes the $k$th zero of the Bessel function $J_\nu$, see  \cite[rel.\ 10.21.2]{nist}. Moreover, $J_\nu(x)>0$ for every $x\in (0,j_\nu)$. Accordingly, $J_\nu$ is negative on the interval $(j_\nu,j_{\nu,2})$. Since $j_{\nu+1}\in (j_\nu,j_{\nu,2})$, the claim follows. 
	
	For $t=0$, relation \eqref{Bessel-inequality} is trivial. Now, let $f(t)=J_\nu(j_{\nu+1}t)-J_{\nu}(j_{\nu+1}) t^\nu$, $t\in (0,1]$; we are going to prove that $f(t)\geq 0$ for every $t\in (0,1]$.  By the recurrence relation from  \cite[rel.\ 10.6.2]{nist} and the fact that $J_{\nu+1}(x)\geq 0$ for every $x\in [0,j_{\nu+1}]$, 
	we have for every $t\in (0,1]$ that
	\begin{align*}
		f'(t)&=j_{\nu+1} J_\nu'(j_{\nu+1}t)-\nu J_{\nu}(j_{\nu+1}) t^{\nu-1}\\&=j_{\nu+1}\left(-J_{\nu+1}(j_{\nu+1}t)+\frac{\nu}{j_{\nu+1}t}J_\nu(j_{\nu+1}t)\right)-\nu J_{\nu}(j_{\nu+1})t^{\nu-1}\\&=-j_{\nu+1}J_{\nu+1}(j_{\nu+1}t)+\frac{\nu}{t}\left(J_\nu(j_{\nu+1}t)-J_{\nu}(j_{\nu+1}) t^\nu\right)\\&=-j_{\nu+1}J_{\nu+1}(j_{\nu+1}t)+\frac{\nu}{t}f(t)\\&\leq \frac{\nu}{t}f(t). 
	\end{align*}
	Therefore, the function $t\mapsto {f(t)}{t^{-\nu}}$ is non-increasing on $(0,1]$; in particular, ${f(t)}{t^{-\nu}}\geq f(1)=0$, which concludes the proof of  \eqref{Bessel-inequality}.   
\end{proof}

	\begin{theorem}\label{Nash-theorem}  Let $N>1$  and $(X,{\sf d},{\sf m})$ be a metric measure space 	
		supporting the Nash  inequality  for some ${\sf C}>0$, i.e., 
		\begin{equation}\label{Nash-metric-space}
			\|u\|_{L^2(X,{\sf m})}\leq {\sf C} \|\nabla u\|_{L^2(X,{\sf m})}^\frac{N}{N+2} \|u\|_{L^1(X,{\sf m})}^\frac{2}{N+2},\ \ \forall u\in  W^{1,2}(X,{\sf m}), 
		\end{equation}
	and assume that for some $x_0\in X$ the  limit 
	$L_N(x_0)$
	exists in \eqref{Limit-exists-N}. Then, 
	\begin{equation}\label{Nash-volume-growth-estimate}
		{\sf AVR}(X)\geq \left(\frac{{\sf CL}_N}{\sf C}\right)^{N+2},
	\end{equation}
where $${\sf CL}_N=\left(\frac{N+2}{2}\right)^\frac{1}{2}\omega_N^{-\frac{1}{N+2}}\left(\frac{N}{2}j^2_\frac{N}{2}\right)^{-\frac{N}{2(N+2)}}.$$   

\end{theorem}

\begin{proof}
	Let $N>1,$ $p=q=2$, $r=1$ and $\theta=\frac{2}{N+2}$; with these choices, the balance condition \eqref{balance-condition} is clearly verified. Moreover,  on the cone $(\mathbb R_+,|\cdot|,{ \sf m}_N)$ 	one has the sharp Nash inequality  
		\begin{equation}\label{new-Nash-1-dim}
		\displaystyle  	\|u\|_{L^2(\mathbb R_+,{ \sf m}_N)}\leq {\sf CL}_N \| u'\|_{L^2(\mathbb R_+,{ \sf m}_N)}^{\frac{N}{N+2}} \|u\|_{L^1(\mathbb R_+,{ \sf m}_N)}^{\frac{2}{N+2}},\ \ \forall  u\in  W^{1,2}(\mathbb R_+,{ \sf m}_N),
	\end{equation}
 and the unique  extremal in \eqref{new-Nash-1-dim} -- up to scaling and multiplicative factor -- is given by
	\begin{equation}\label{new-extrm-nash--1}
		u_0(t)=\ds\left\{ 
		\begin{array}{lll}
			1-\frac{t^{-\nu}}{J_{\nu}(j_{\nu+1})} J_{\nu}\left(j_{\nu+1}t\right),\ &\text {if}&  t<1; \\
			
			0, &\text {if}&  t\geq 1,
		\end{array}\right.
	\end{equation}
	where $\nu=\frac{N}{2}-1,$ see by Carlen and Loss \cite{Carlen-Loss}. Note that \eqref{new-Nash-1-dim} has been established for $N\in \mathbb N\setminus \{1\}$, but a closer inspection of the proof in \cite{Carlen-Loss} shows that it holds for every $N>1.$
	Accordingly,  supp$(u_0)=[0,1]$, and we may choose $R_0=1$; by construction, $u_0(1)=0$ and by Proposition \ref{Bessel-lemma} it follows that $u_0\geq 0$ and  $u_0'\leq 0$ on $[0,1]$ with $u_0'(1)=0$. Therefore,  we may apply Theorem \ref{main-theorem-unified}. 
\end{proof}

Theorem \ref{main-theorem-CD} together with the latter proof  implies the following result (which in turn, gives also Theorem \ref{main-theorem-Nash}):

\begin{theorem}\label{Nash-theorem-CD}  Let $N>1$  and $(X,{\sf d},{\sf m})$ be an essentially non-branching  ${\sf CD}(0,n)$ metric measure space with ${\sf AVR}(X) >0.$ Then one has
	\begin{equation}\label{Nash-metric-space-CD}
		\|u\|_{L^2(X,{\sf m})}\leq {\sf AVR}(X)^{-\frac{1}{N+2}} {\sf CL}_N \|\nabla u\|_{L^2(X,{\sf m})}^\frac{N}{N+2} \|u\|_{L^1(X,{\sf m})}^\frac{2}{N+2},\ \ \forall u\in  W^{1,2}(X,{\sf m}), 
	\end{equation}
	and  ${\sf AVR}(X)^{-\frac{1}{N+2}} {\sf CL}_N$ is sharp. 
	
	Moreover, equality holds in \eqref{Nash-metric-space-CD} for some non-zero  $u\in  W^{1,2}(X,{\sf m})$   if and only if 
	$X$ is an $N$-volume cone with a tip $x_0\in X$ and up to scaling and multiplicative factor,  one has that $u(x)=u_0\left({\sf AVR}(X)^{\frac{1}{N}}{\sf d}(x_0,x)\right)$  for ${\sf m}$-a.e.\  $x\in B_{x_0}({\sf AVR}(X)^{-\frac{1}{N}})$, where $u_0$ is the function from \eqref{new-extrm-nash--1}.
\end{theorem}

\subsubsection{Faber--Krahn  inequality II}
The second Faber--Krahn inequality concerns the classical  Dirichlet $p$-eigenvalue problem. Let $q=p>1$ and $r\to 0$ in the Gigliardo--Nirenberg inequality \eqref{GN-1-dimensional-model}, which reduces to 
\begin{equation}\label{Krahn-inq-1-dim}
		\|u\|_{L^p(\mathbb R_+,{\sf m}_{N})}\leq {\sf K}_{\rm opt} \|u'\|_{L^p(\mathbb R_+,{\sf m}_{N})} {\sf m}_N({\rm supp}(u))^{\frac{1}{N}},\ \ \forall u\in W_{\rm loc}^{1,p}(\mathbb R_+,{\sf m}_N)\cap L^p(\mathbb R_+,{\sf m}_N).
\end{equation}
Indeed, by the balance condition \eqref{balance-condition}, one has that $\theta\to 1$ as $r\to 0$, and 
 \begin{align*}
	 	\lim_{r\to 0}\|u\|_{L^{r}(\mathbb R_+,{ \sf m}_N)}^{1-\theta}={\sf m}_N({\rm supp}(u))^{\frac{1}{N}}.
	 \end{align*}
In addition, a standard compactness argument implies that the best constant $\mathcal K_{p,N}:={\sf K}_{\rm opt}$ in \eqref{Krahn-inq-1-dim} is achieved by a function $u_0:[0,1]\to \mathbb R_+$ with supp$(u_0)=[0,1]$  and solving the ODE:
$$\left(|u_0'(\rho)|^{p-2}u_0'(\rho)\rho^{N-1}\right)'+\left(\mathcal K_{p,N}\omega_N^{1/N}\right)^{-p}u_0^{p-1}(\rho)\rho^{N-1}=0,\ \ \rho\in (0,1).$$ 
In particular, if $p=2$, one has that $\mathcal K_{2,N}=\left(j_{\nu}\omega_N^{1/N}\right)^{-1}$ and up to a multiplicative constant, $u_0(t)=t^{-\nu}J_\nu(j_\nu t)$, $t\in (0,1),$ where $\nu=\frac{N}{2}-1.$ Similar results to Theorems \ref{GN-1-theorem} and \ref{GN-1-theorem-CD} can be also deduced   in the present setting. 

%


\section{Genuine borderline  cases}\label{section-borderline-0}

This section contains two genuinely different borderline cases for the Gagliardo--Nirenberg--Sobolev interpolation inequality; namely, the logarithmic-Sobolev  and  Moser--Trudinger inequalities. 

\subsection{Logarithmic-Sobolev  inequality} It is well-known that both    \eqref{gn-elso} and \eqref{gn-masodik} reduce to the sharp logarithmic-Sobolev inequality whenever $\alpha\to 1$. This inequality states that  
for every $u\in W^{1,p}(\mathbb R_+,{ \sf m}_N)$ with 
$\|u\|_{L^p(\mathbb R_+,{\sf m}_N)}=1$, one has
\begin{equation}\label{log-Sob-1D}
\int_{\mathbb R_+}\vert u\vert^p\log\vert u\vert^p{\rm d}{\sf m}_N\le \frac{N}{p}\log\bigg(\mathcal{L}_{p,N}\|u'\|^p_{L^p(\mathbb R_+,{\sf m}_N)}\bigg),
\end{equation}
where
\begin{equation}	\label{explicit-L}
	\mathcal
	L_{p,N}={
		\frac{p}{N}\left(\frac{p-1}{e}\right)^{p-1}\left(\omega_N{\Gamma\left(\frac{N}{p'}+1\right)}\right)^{-\frac{p}{N}}}, 
\end{equation}
and the unique extremal function, up to scaling, is the Gaussian 
\begin{equation}\label{gaussian-1-dim}
	u_0(t)=
	\omega_N^{-\frac{1}{p}}\Gamma\left(\frac{N}{p'}+1\right)^{-\frac{1}{p}}
	e^{-\frac{t^{p'}}{p}},\ t\geq 0.
\end{equation}
Note  that \eqref{log-Sob-1D} is well-known from Del Pino and Dolbeault \cite{delPino} for $p<N$ and $N\in \mathbb N$, extended to the general case $p,N>1$ by Balogh, Don and Krist\'aly \cite{BDK-TAMS} via optimal mass transportation. 
	
	\begin{theorem}\label{log-sob-2-theorem}   Let $N,p>1$  and $(X,{\sf d},{\sf m})$ be a metric measure space 	
		supporting for some ${\sf C}>0$ the  logarithmic-Sobolev inequality$:$ for every  $u\in W^{1,p}(X,{ \sf m})$ with 
		$\|u\|_{L^p(X,{\sf m})}=1$, one has
		\begin{equation}\label{log-sobo-2-metric-space}
		\int_{X}\vert u\vert^p\log\vert u\vert^p{\rm d}{\sf m}\le \frac{N}{p}\log\bigg({\sf C}\|\nabla u\|^p_{L^p(X,{\sf m})}\bigg). 
		\end{equation}
		If the  limit $L_N(x_0)$
		exists in \eqref{Limit-exists-N} for some $x_0\in X$,   then  
		\begin{equation}\label{log-sob-volume-growth-estimate}
			{\sf AVR}(X)\geq \left(\frac{\mathcal{L}_{p,N}}{{\sf C}}\right)^\frac{N}{p}.
		\end{equation}
	\end{theorem}
	\begin{proof}
		We may assume that $L_N(x_0)<\infty$; otherwise, we have nothing to prove. For every $R>0$, we consider the function 
		$u_R(x)=u_0\left(\frac{{\sf d}(x_0,x)}{R}\right),\ \ x\in X,$
		where $u_0\in  W^{1,p}(\mathbb R_+,{\sf m}_N)$ is the Gaussian from \eqref{gaussian-1-dim}, i.e., 
		$\|u_0\|_{L^p(\mathbb R_+,{\sf m}_N)}=1$ and 
		\begin{equation}\label{kesobb-log-Sob-1D}
			\int_{\mathbb R_+} u_0^p\log u_0^p{\rm d}{\sf m}_N = \frac{N}{p}\log\bigg(\mathcal{L}_{p,N}\|u_0'\|^p_{L^p(\mathbb R_+,{\sf m}_N)}\bigg).
		\end{equation} 
		
		We notice that $u_R\in W^{1,p}(X,{ \sf m})$ for every $R>0.$ Indeed, due to  $L_N(x_0)<\infty$ and the fast decay property of the Gaussian at infinity, one can proceed similarly as in \eqref{r-0-construction} and \eqref{derivative-oo00} to show that  $(u_0^p)'(\rho){\sf m}(B_{x_0}(R\rho))\in L^1(\mathbb R_+)$ and $ (|u_0'|^p)'(\rho){\sf m}(B_{x_0}(R\rho))\in L^1(\mathbb R_+)$, respectively. Therefore, by using Proposition \ref{proposition-co-area} and \eqref{grad-estim-R}, it follows that 
		\begin{equation}\label{1-kell-kesobb}
			\|u_R\|_{L^p(X,{\sf m})}^p=\int_X u_R^p{\rm d}{\sf m}=\int_X u_0^p\left(\frac{{\sf d}(x_0,x)}{R}\right){\rm d}{\sf m}=-\int_{\mathbb R_+}(u_0^p)'(\rho){\sf m}(B_{x_0}(R\rho)){\rm d}\rho,
		\end{equation}
	and 
		\begin{equation}\label{2-kell-kesobb}
		\|\nabla u_R\|_{L^p(X,{\sf m})}^p\leq \frac{1}{R^p}\int_X |u_0'|^p\left(\frac{{\sf d}(x_0,x)}{R}\right){\rm d}{\sf m}=-\frac{1}{R^p}\int_{\mathbb R_+} (|u_0'|^p)'(\rho){\sf m}\left(B_{x_0}(R\rho)\right){\rm d}\rho.
	\end{equation}
In addition, by Proposition \ref{proposition-co-area},	
 it also follows that 	
\begin{equation}\label{3-kell-kesobb}
	\int_{X} u_R^p\log  u_R^p{\rm d}{\sf m}=-\int_{\mathbb R_+}(u_0^p\log u_0^p)'(\rho){\sf m}\left(B_{x_0}(R\rho)\right){\rm d}\rho.
\end{equation}		
By relations \eqref{1-kell-kesobb}, \eqref{2-kell-kesobb}, \eqref{3-kell-kesobb} and Lebesgue dominated convergence theorem, it follows that
\begin{align}\label{1-1-kell}
\nonumber	\lim_{R\to \infty}\frac{\|u_R\|_{L^p(X,{\sf m})}^p}{R^N}&=-{\sf AVR}(X)\omega_N\int_{\mathbb R_+}(u_0^p)'(\rho)\rho^N{\rm d}\rho={\sf AVR}(X)\omega_NN\int_{\mathbb R_+}u_0^p(\rho)\rho^{N-1}{\rm d}\rho\\&={\sf AVR}(X)\|u_0\|_{L^p(\mathbb R_+,{\sf m}_N)}^p={\sf AVR}(X),
\end{align}
\begin{align}\label{2-1-kell}
	\nonumber	\lim_{R\to \infty}\frac{\|\nabla u_R\|_{L^p(X,{\sf m})}^p}{R^{N-p}}&\leq -{\sf AVR}(X)\omega_N\int_{\mathbb R_+} (|u_0'|^p)'(\rho)\rho^N{\rm d}\rho={\sf AVR}(X)\omega_NN\int_{\mathbb R_+}|u_0'|^p(\rho)\rho^{N-1}{\rm d}\rho\\&={\sf AVR}(X)\|u_0'\|_{L^p(\mathbb R_+,{\sf m}_N)}^p,
\end{align}
and 
\begin{align}\label{3-1-kell}
	\nonumber	\lim_{R\to \infty}\frac{\ds\int_{X} u_R^p\log  u_R^p{\rm d}{\sf m}}{R^N}&=-{\sf AVR}(X)\omega_N\int_{\mathbb R_+}(u_0^p\log u_0^p)'(\rho)\rho^N{\rm d}\rho={\sf AVR}(X)\omega_NN\int_{\mathbb R_+}u_0^p\log u_0^p\rho^{N-1}{\rm d}\rho\\&={\sf AVR}(X)\int_{\mathbb R_+}u_0^p\log u_0^p{\rm d}{\sf m}_N.
\end{align}

According to  properties \eqref{1-kell-kesobb} and \eqref{2-kell-kesobb}, we may use  $$\tilde u_R=\frac{u_R}{\|u_R\|_{L^p(X,{\sf m})}}\in W^{1,p}(X,{ \sf m})$$
as a test function in \eqref{log-sobo-2-metric-space}. Indeed, since $\|\tilde u_R\|_{L^p(X,{\sf m})}=1$, one has that 
			\begin{equation}\label{log-sobo-2-metric-space-99}
			\int_{X} \tilde u_R^p\log \tilde u_R^p{\rm d}{\sf m}\le \frac{N}{p}\log\bigg({\sf C}\|\nabla \tilde u_R\|^p_{L^p(X,{\sf m})}\bigg). 
		\end{equation}
Since 
	$$\int_{X} \tilde u_R^p\log \tilde u_R^p{\rm d}{\sf m}=\frac{1}{\|u_R\|_{L^p(X,{\sf m})}^p}\left(\int_{X} u_R^p\log  u_R^p{\rm d}{\sf m}-\|u_R\|_{L^p(X,{\sf m})}^p\log\|u_R\|_{L^p(X,{\sf m})}^p\right),$$	
	relation  \eqref{log-sobo-2-metric-space-99} implies that
	\begin{equation}\label{00-11-log}
			\frac{1}{\|u_R\|_{L^p(X,{\sf m})}^p}\int_{X} u_R^p\log  u_R^p{\rm d}{\sf m}-\log\|u_R\|_{L^p(X,{\sf m})}^p \leq \frac{N}{p}\log\bigg({\sf C}\frac{\|\nabla  u_R\|^p_{L^p(X,{\sf m})}}{\|u_R\|_{L^p(X,{\sf m})}^p}\bigg).
	\end{equation}
Writing \eqref{00-11-log} equivalently into 
	$$
	\frac{R^N}{\|u_R\|_{L^p(X,{\sf m})}^p}\int_{X} \frac{u_R^p\log  u_R^p}{R^N}{\rm d}{\sf m}-\log\frac{\|u_R\|_{L^p(X,{\sf m})}^p}{R^N} \leq \frac{N}{p}\log\bigg({\sf C}\frac{\|\nabla  u_R\|^p_{L^p(X,{\sf m})}}{R^{N-p}}	\frac{R^N}{\|u_R\|_{L^p(X,{\sf m})}^p}\bigg),
$$
and letting $R\to \infty$ in the latter inequality, by relations \eqref{1-1-kell}, \eqref{2-1-kell}, \eqref{3-1-kell}, it yields that
		$$\int_{\mathbb R_+}u_0^p\log u_0^p{\rm d}{\sf m}_N-\log {\sf AVR}(X)\leq \frac{N}{p}\log\left({\sf C}\|u_0'\|_{L^p(\mathbb R_+,{\sf m}_N)}^p\right).$$
Combining the latter relation with  \eqref{kesobb-log-Sob-1D}, we obtain that 
$
	{\sf AVR}(X)\geq \left(\frac{\mathcal{L}_{p,N}}{{\sf C}}\right)^\frac{N}{p},$
	which is precisely the desired relation \eqref {log-sob-volume-growth-estimate}.
	\end{proof}

The sharp logarithmic Sobolev inequality on ${\sf CD}(0,N)$ spaces has been already established by Balogh, Krist\'aly and Tripaldi \cite{BKT}, the equality case being  discussed by Nobili and Violo \cite{NV-new}; for  completeness, we summarize its most general form.   

\begin{theorem}\label{log-sob-2-theorem-CD} {\rm (see \cite{BKT, NV-new})}   Let $N,p>1$  and $(X,{\sf d},{\sf m})$ be an essentially non-branching  ${\sf CD}(0,N)$ metric measure space with ${\sf AVR}(X) >0.$ Then for any  $u\in W^{1,p}(X,{ \sf m})$ with 
	$\|u\|_{L^p(X,{\sf m})}=1$, one has
	\begin{equation}\label{log-sobo-2-metric-space-CD}
		\int_{X}\vert u\vert^p\log\vert u\vert^p{\rm d}{\sf m}\le \frac{N}{p}\log\bigg({\sf AVR}(X)^{-\frac{p}{N}}\mathcal
		L_{p,N}\|\nabla u\|^p_{L^p(X,{\sf m})}\bigg), 
	\end{equation}
	and the constant ${\sf AVR}(X)^{-\frac{p}{N}}\mathcal
	L_{p,N}$ is sharp, where $\mathcal
	L_{p,N}$ is from  \eqref{explicit-L}. 
	
	Moreover, equality holds in \eqref{log-sobo-2-metric-space-CD} for some non-zero, non-negative  $u\in  W^{1,p}(X,{\sf m})$ if and only if 
	$X$ is an $N$-volume cone with a tip $x_0\in X$ and up to scaling, one has that $u(x)=\left(\Gamma(\frac{N}{p'}+1)\omega_N{\sf AVR}(X)\right)^{-\frac{1}{p}}e^{-\frac{{\sf d}^{p'}(x_0,x)}{p}}$  for ${\sf m}$-a.e.\ $x\in X$.
\end{theorem}

\subsection{Moser--Trudinger inequality}\label{section-MT}
Another limit case in Sobolev inequalities is the Moser--Tru\-dinger inequality. Indeed, if $\Omega\subset \mathbb R^n$ is an open  bounded  set, the Sobolev space $W_0^{1,p}(\Omega)$ can be continuously embedded into $L^q(\Omega)$ for every $q\in [1,\frac{np}{n-p}]$ whenever $1<p<n.$ Although expected, when $p\to n$, the space $W_0^{1,n}(\Omega)$ cannot be embedded into $L^\infty(\Omega)$. However, according to Trudinger \cite{Trudinger}, $W_0^{1,p}(\Omega)$ can be embedded into the Orlicz space $L_{\phi_n}(\Omega)$ for the Young function $\phi_n(s)=e^{\alpha|s|^\frac{n}{n-1}}-1$, $s\in \mathbb R$, for $\alpha>0$ sufficiently small. In fact, Moser \cite{Moser} proved that  
for every open set $\Omega\subset \mathbb R^n$ $(n\geq 2)$ with finite Lebesgue-measure, one has that 
\begin{equation}\label{supremum-constant-MT}
	\sup\left\{\frac{1}{\mathcal L^n(\Omega)}\int_\Omega e^{\alpha|u|^\frac{n}{n-1}}{\rm d}x:u\in W_0^{1,n}(\Omega),\|\nabla u\|_{L^n(\Omega)}\leq 1\right\}=\left\{
	\begin{array}{lll}
		C(\alpha,n)<+\infty &{\rm
			if} & 0<\alpha\leq \alpha_n;
		\\ +\infty& {\rm
			if} & \alpha>\alpha_n,
	\end{array}
	\right.
\end{equation}
where $\alpha_n=n\sigma_{n-1}^\frac{1}{n-1}$ is the critical exponent in the Moser--Trudinger inequality, the value $\sigma_{n-1}$ being the surface area of the unit
ball in $\mathbb R^n.$ It has been proved first that the above supremum is attained for $\alpha=\alpha_n$ and for the unit ball $\Omega=B_0(1)\subset \mathbb R^n$, see Carleson and Chang \cite{Carleson-Chang}, and later for $\alpha=\alpha_n$ and arbitrary domains $\Omega\subset \mathbb R^n$, see Lin \cite{Lin-TAMS}. 

For further use,  we recall  the Moser function  $\tilde w_k(x)=w_k(|x|),$ $k\in \mathbb N$, $x\in B_0(1)\subset \mathbb R^n$, where
\begin{equation}\label{Moser-function}
	w_k(t)=\frac{1}{\sigma_{n-1}^{1/n}}\left\{
	\begin{array}{lll}
		(\log k)^\frac{n-1}{n}, &{\rm
			if} & t\in [0,\frac{1}{k});
		\\ \frac{\log(1/t)}{{(\log k)}^{{1}/{n}}},& {\rm
			if} & \frac{1}{k}\leq t\leq 1;\\
		0,&{\rm
			if} &  t> 1.
	\end{array}
	\right.
\end{equation}
A simple argument shows that supp$\,(\tilde w_k)=\overline B_0(1)$,   $\|\nabla \tilde w_k\|_{L^n(B_0(1))}=\| w_k'\|_{L^n((0,1),{\sf m}_n)}=1$ and  $\tilde w_k\in W_0^{1,n}(B_0(1))$ for every $k\in \mathbb N$, where, as usual, ${\sf m}_n=n\omega_nr^{n-1}\mathcal L^1$. Moreover,  for every  $\alpha>\alpha_n$, one has that 
\begin{equation}\label{moser-vegtelen}
\lim_{k\to \infty}	\int_{B_0(1)} e^{\alpha \tilde w_k^\frac{n}{n-1}}{\rm d}x=\lim_{k\to \infty}	\int_0^1 e^{\alpha  w_k^\frac{n}{n-1}}{\rm d}{\sf m}_n=+\infty.
\end{equation}

The volume growth result under the validity of the Moser--Trudinger inequality reads as follows: 

	\begin{theorem}\label{MT-theorem}   Let $n\geq 2$ be an integer  and $(X,{\sf d},{\sf m})$ be a metric measure space 	
	supporting the following Moser--Trudinger inequality$:$ there exists ${\sf C}>0$ such that for every open set $\Omega\subset X$ with finite ${\sf m}$-measure one has  
	\begin{equation}\label{MT-metric-space}
		\sup\left\{\frac{1}{{\sf m}(\Omega)}\int_\Omega e^{{\sf C}|u|^\frac{n}{n-1}}{\rm d}{\sf m}:u\in W_0^{1,n}(\Omega,{\sf m}),\|\nabla u\|_{L^n(\Omega,{\sf m})}\leq 1\right\}<\infty.
	\end{equation}
	Assume that for some $x_0\in X$ the  limit  $L_n(x_0)$
	exists in \eqref{Limit-exists-N}. Then  
	\begin{equation}\label{MT-volume-growth-estimate}
		{\sf AVR}(X)\geq \left(\frac{{\sf C}}{\alpha_n}\right)^{n-1}.
	\end{equation}
\end{theorem}
\begin{proof} If $L_n(x_0)=\infty$, we have nothing to prove; thus, we consider that $L_n(x_0)<\infty$. Assume by contradiction that ${\sf AVR}(X):=L_n(x_0)< \left(\frac{{\sf C}}{\alpha_n}\right)^{n-1}$; in particular, one can find a small  $\epsilon>0$ such that 
	\begin{equation}\label{eps-valasztas}
		{\sf AVR}(X)+\epsilon< \left(\frac{{\sf C}}{\alpha_n}\right)^{n-1}.
	\end{equation}
By using the Moser function from \eqref{Moser-function},  we consider for every  $R>0$ and $k\in \mathbb N$ the function $$u_{R,\epsilon,k}(x)=\frac{1}{({\sf AVR}(X)+\epsilon)^\frac{1}{n}}w_k\left(\frac{{\sf d}(x_0,x)}{R}\right),\ \ x\in B_{x_0}(R).$$
	 A similar computation as in the proof of Theorem \ref{main-theorem-unified} yields  that 
\begin{align*}
	\nonumber	\|\nabla u_{R,\epsilon,k}\|^n_{L^n(B_{x_0}(R),{\sf m})}&\leq \frac{1}{{\sf AVR}(X)+\epsilon} \frac{1}{R^n}\int_{B_{x_0}(R)} |w_k'|^n\left(\frac{{\sf d}(x_0,x)}{R}\right){\rm d}{\sf m}\\&=\frac{1}{{\sf AVR}(X)+\epsilon}\frac{1}{R^n}\left(\ds |w_k'|^n(1){\sf m}\left(B_{x_0}(R)\right)	-\ds\int_0^1 (|w_k'|^n)'(\rho){\sf m}\left(B_{x_0}(R\rho)\right){\rm d}\rho\right).
\end{align*}
Letting $R\to \infty$, by the Lebesgue dominated convergence theorem   one has that 
\begin{align*}\lim_{R\to \infty}\|\nabla u_{R,\epsilon,k}\|^n_{L^n(B_{x_0}(R),{\sf m})}&\leq \frac{\omega_n{\sf AVR}(X)}{{\sf AVR}(X)+\epsilon}\left(\ds |w_k'|^n(1)	-\ds\int_0^1 (|w_k'|^n)'(\rho)\rho^n{\rm d}\rho\right)\\&=\frac{{\sf AVR}(X)}{{\sf AVR}(X)+\epsilon}n\omega_n\ds \int_0^1 |w_k'|^n(\rho)\rho^{n-1}{\rm d}\rho=\frac{{\sf AVR}(X)}{{\sf AVR}(X)+\epsilon}\ds \int_0^1 |w_k'|^n{\rm d}{\sf m}_n\\&= \frac{{\sf AVR}(X)}{{\sf AVR}(X)+\epsilon}<1.
\end{align*}
In particular, there exists $\tilde R>0$ such that $\|\nabla u_{R,\epsilon,k}\|_{L^n(B_{x_0}(R),{\sf m})}\leq 1$ for every $R>\tilde R$ and $k\in \mathbb N$. In addition, the definition of $u_{R,\epsilon,k}$ and the above estimate show  that  $u_{R,\epsilon,k}\in W_0^{1,n}(B_{x_0}(R),{\sf m})$. Thus, if we denote by $S\in (0,\infty)$ the supremum in the assumption \eqref{MT-metric-space}, one has for every $R>\tilde R$ and $k\in \mathbb N$  that 
\begin{equation}\label{I-definicio-}
	I_{R,\epsilon,k}:=\frac{1}{{\sf m}(B_{x_0}(R))}\int_{B_{x_0}(R)} e^{{\sf C}u_{R,\epsilon,k}^\frac{n}{n-1}}{\rm d}{\sf m}\leq  S.
\end{equation}
Let $\alpha:=\frac{{\sf C}}{({\sf AVR}(X)+\epsilon)^\frac{1}{n-1}}$; with this notation,  by using the fact that $w_k(1)=0$ and Proposition \ref{proposition-co-area}, one has that
\begin{align*}
	I_{R,\epsilon,k}&=\frac{1}{{\sf m}(B_{x_0}(R))}\int_{B_{x_0}(R)} e^{\alpha w_{k}^\frac{n}{n-1}\left(\frac{{\sf d}(x_0,x)}{R}\right)}{\rm d}{\sf m}\\&=\frac{1}{{\sf m}(B_{x_0}(R))}\left({\sf m}(B_{x_0}(R))-\int_0^1 \left(e^{\alpha w_{k}^\frac{n}{n-1}\left(\rho\right)}\right)'{\sf m}(B_{x_0}(R\rho)){\rm d}\rho\right).
\end{align*}
Taking $R\to \infty$, by the Lebesgue dominated convergence theorem   and an integration by parts one has for every $k\in \mathbb N$ that 
$$\lim_{R\to \infty}I_{R,\epsilon,k}=1-\int_0^1 \left(e^{\alpha w_{k}^\frac{n}{n-1}\left(\rho\right)}\right)'\rho^n{\rm d}\rho=n\int_0^1 e^{\alpha w_{k}^\frac{n}{n-1}\left(\rho\right)}\rho^{n-1}{\rm d}\rho=\frac{1}{\omega_n}\int_0^1 e^{\alpha  w_k^\frac{n}{n-1}}{\rm d}{\sf m}_n.$$
Combining the latter relation with \eqref{I-definicio-}, it follows for every $k\in \mathbb N$ that
$\ds\int_0^1 e^{\alpha  w_k^\frac{n}{n-1}}{\rm d}{\sf m}_n\leq \omega_n S.$
Letting $k\to \infty$ in this inequality and taking into account that $\alpha=\frac{{\sf C}}{({\sf AVR}(X)+\epsilon)^\frac{1}{n-1}}>\alpha_n$, see \eqref{eps-valasztas}, we obtain a contradiction to \eqref{moser-vegtelen}. The proof is complete. 
\end{proof}

%
%
%
%
%


At the end of this section, we focus on the Moser--Trudinger inequality on ${\sf CD}(0,n)$ spaces, $n\in \mathbb N;$ to do this, let 
$${\sf MT}_n=C(\alpha_n,n),$$
where the constant $C(\alpha,n)$ is from \eqref{supremum-constant-MT} for  $\alpha\leq \alpha_n$, and $\alpha_n$ is the  critical exponent in the Euclidean Moser--Trudinger inequality. In addition, let $u_0:[0,1]\to \mathbb R_+$ be  the profile function of the extremizer in the Carleson--Chang problem, i.e., the function $u(x)=u_0(|x|),$ $x\in B_0(1)$, achieves the supremum in \eqref{supremum-constant-MT} for $\alpha=\alpha_n$ on  the unit ball $B_0(1)\subset \mathbb R^n$. 

	\begin{theorem}\label{MT-theorem-CD}   Let $n\geq 2$ be an integer  and $(X,{\sf d},{\sf m})$ be an essentially non-branching  ${\sf CD}(0,n)$ metric measure space with ${\sf AVR}(X) >0.$ Then for every open bounded set $\Omega\subset X$  one has  
	\begin{equation}\label{MT-metric-space-CD}
		\sup\left\{\frac{1}{{\sf m}(\Omega)}\int_\Omega e^{ {\sf AVR}(X)^\frac{1}{n-1} \alpha_n |u|^\frac{n}{n-1}}{\rm d}{\sf m}:u\in W_0^{1,n}(\Omega,{\sf m}),\|\nabla u\|_{L^n(\Omega,{\sf m})}\leq 1\right\}\leq {\sf MT}_n,
	\end{equation}
	where the constant ${\sf AVR}(X)^\frac{1}{n-1}\alpha_n $ is sharp. In addition, we have the following statements$:$ 
	
	\begin{itemize}
		\item[(i)] If there exist an  open bounded set $\Omega\subset X$ and a non-negative function  $u\in W_0^{1,n}(\Omega,{\sf m})\setminus\{0\}$ such that $\|\nabla u\|_{L^n(\Omega,{\sf m})}\leq 1$ and 
		\begin{equation}\label{MT-egyenloseg}
			\frac{1}{{\sf m}(\Omega)}\int_\Omega e^{ {\sf AVR}(X)^\frac{1}{n-1} \alpha_n u^\frac{n}{n-1}}{\rm d}{\sf m}={\sf MT}_n,
		\end{equation}
		then $X$ is an $n$-volume cone with a tip $x_0\in X$, $u$ is $x_0$-radial and the set  $\Omega$ is the ball $B_{x_0}({\sf AVR}(X)^{-1/n}R)$ up to an ${\sf m}$-negligible set, where $R=({\sf m}(\Omega)\omega_n^{-1})^{1/n};$ 
		
		\item[(ii)] If $X$ is an $n$-volume cone with a tip $x_0\in X,$ then for every $R>0$ the function $u(x)={\sf AVR}(X)^{-1/n}u_0({\sf AVR}(X)^{1/n}R^{-1}{\sf d}(x_0,x)),$ $x\in \Omega= B_{x_0}({\sf AVR}(X)^{-1/n}R)$, verifies relation \eqref{MT-egyenloseg}, where $u_0$ is the profile function of the extremizer in the Carleson--Chang problem. 
	\end{itemize}

\end{theorem}

\begin{proof} First of all, we notice that by using fine properties of Green functions, a slightly different relation to \eqref{MT-metric-space-CD} has been established recently by 
	Fontana, Morpurgo and Qin \cite{FMQ} on Riemannian manifolds with non-negative Ricci curvature. Instead of adapting their argument, we provide a self-contained proof of \eqref{MT-metric-space-CD} which will be also crucial for discussing the equality case; note that no equality case has been considered in \cite{FMQ}.
		We divide the proof into four steps. 
	
	\textit{\underline{Step 1: proof of \eqref{MT-metric-space-CD}}.}  
		If we use \eqref{supremum-constant-MT} for (Euclidean) radial functions $u_r(x)=\tilde u(|x|)$, $x\in B_0(R)$, $R>0,$  we obtain that for every $\tilde u\in W_0^{1,n}([0,R),{\sf m}_n)$ with  $\|\tilde u'\|_{L^n([0,R),{\sf m}_n)}\leq 1$ one has
		\begin{equation}\label{extremal-MT}
		\frac{1}{\omega_n R^n}\int_0^R e^{\alpha_n |\tilde u|^\frac{n}{n-1}}{\rm d}{\sf m}_n\leq {\sf MT}_n.
	\end{equation}

	Let us fix an open bounded set $\Omega\subset X$ and a function $u\in W_0^{1,n}(\Omega,{\sf m})$ with $\|\nabla u\|_{L^n(\Omega,{\sf m})}\leq 1$; without the loss of generality, we may assume that $u$ is non-negative on $\Omega.$ Due to the P\'olya--Szeg\H o inequality \eqref{Polya-Szego}, it follows that 
	\begin{equation}\label{equation-AVR-u-derivalt}
		{\sf AVR}(X)^\frac{1}{n}\|(u^*)'\|_{L^n(\mathbb R_+,{\sf m}_n)}\leq 1.
	\end{equation}
	  In particular, we may apply \eqref{extremal-MT} for $\tilde u={\sf AVR}(X)^\frac{1}{n}u^*$ on $[0,R)$ with $\omega_n R^n={\sf m}_n([0,R))={\sf m}(\Omega)$, obtaining through the Cavalieri principle \eqref{Cavallieri} that
	\begin{equation}\label{MT-estimate-fontos}
			\frac{1}{{\sf m}(\Omega)}\int_\Omega e^{ {\sf AVR}(X)^\frac{1}{n-1} \alpha_n u^\frac{n}{n-1}}{\rm d}{\sf m}=\frac{1}{{\sf m}(\Omega)}\int_\Omega e^{  \alpha_n ({\sf AVR}(X)^\frac{1}{n}u)^\frac{n}{n-1}}{\rm d}{\sf m}=\frac{1}{\omega_n R^n}\int_0^R e^{\alpha_n \tilde u^\frac{n}{n-1}}{\rm d}{\sf m}_n\leq {\sf MT}_n,
	\end{equation}
	which proves \eqref{MT-metric-space-CD}.

\textit{\underline{Step 2: sharpness of \eqref{MT-metric-space-CD}}.}
By contradiction, we assume that there exists ${\sf C}>0$ such that ${\sf C}>{\sf AVR}(X)^\frac{1}{n-1}\alpha_n$ and for every open bounded set $\Omega\subset X$  one has  
	$$	\sup\left\{\frac{1}{{\sf m}(\Omega)}\int_\Omega e^{ {\sf C} |u|^\frac{n}{n-1}}{\rm d}{\sf m}:u\in W_0^{1,n}(\Omega,{\sf m}),\|\nabla u\|_{L^n(\Omega,{\sf m})}\leq 1\right\}<\infty.$$
	Thus, by Theorem \ref{MT-theorem}, we have that ${\sf AVR}(X)\geq \left(\frac{{\sf C}}{\alpha_n}\right)^{n-1},$  contradicting  ${\sf C}>{\sf AVR}(X)^\frac{1}{n-1}\alpha_n$. 
	
\textit{\underline{Step 3: proof of ${\rm (i)}$.}}	Let us fix an  open bounded set $\Omega\subset X$ and a non-negative function  $u\in W_0^{1,n}(\Omega,{\sf m})\setminus\{0\}$ with  $\|\nabla u\|_{L^n(\Omega,{\sf m})}\leq 1$ and such that \eqref{MT-egyenloseg} holds. We first prove that  
\begin{equation}\label{MT-gradiens=1}
	\|\nabla u\|_{L^n(\Omega,{\sf m})}= 1.
\end{equation}
Indeed, by contradiction, let us assume that $a:=\|\nabla u\|_{L^n(\Omega,{\sf m})}\in [0,1)$. The case $a=0$ trivially leads us to a contradiction, thus we assume that $a\in (0,1)$. Then for the function $w_a:=u/a\in W_0^{1,n}(\Omega,{\sf m})$ we have $\|\nabla w_a\|_{L^n(\Omega,{\sf m})}= 1$, thus by \eqref{MT-egyenloseg} and \eqref{MT-estimate-fontos} one has 
	$${\sf MT}_n=\frac{1}{{\sf m}(\Omega)}\int_\Omega e^{ {\sf AVR}(X)^\frac{1}{n-1} \alpha_n u^\frac{n}{n-1}}{\rm d}{\sf m}<\frac{1}{{\sf m}(\Omega)}\int_\Omega e^{ {\sf AVR}(X)^\frac{1}{n-1} \alpha_n w_a^\frac{n}{n-1}}{\rm d}{\sf m}\leq {\sf MT}_n,$$
	a contradiction.  Thus \eqref{MT-gradiens=1} holds. 
	
	For the function $u$ in the assumption \eqref{MT-egyenloseg}, we  should have equality in  \eqref{MT-estimate-fontos}, i.e., 
	\begin{equation}\label{MT-tilde-egyenlet}
			\frac{1}{\omega_n R^n}\int_0^R e^{\alpha_n \tilde u^\frac{n}{n-1}}{\rm d}{\sf m}_n= {\sf MT}_n,
	\end{equation}
	where $\omega_n R^n={\sf m}_n([0,R))={\sf m}(\Omega)$ and $\tilde u={\sf AVR}(X)^\frac{1}{n}u^*$ on $[0,R)$. We prove that
	\begin{equation}\label{gradiens=2=MT-1-dim}
		 \|\tilde u'\|_{L^n([0,R),{\sf m}_n)}= 1.
	\end{equation}
	Indeed, by \eqref{equation-AVR-u-derivalt} one has that $\|\tilde u'\|_{L^n([0,R),{\sf m}_n)}\leq  1$. Let us assume by contradiction that $a:=\|\tilde u'\|_{L^n([0,R),{\sf m}_n)}\in [0,1)$. Again, the case $a=0$ is trivial. If $a>0$, then for $\tilde w_a:=\tilde u/a$ we have  $\|\tilde w_a'\|_{L^n([0,R),{\sf m}_n)}= 1$ and according to \eqref{MT-tilde-egyenlet} and \eqref{extremal-MT} one has 
	$${\sf MT}_n=\frac{1}{\omega_n R^n}\int_0^R e^{\alpha_n \tilde u^\frac{n}{n-1}}{\rm d}{\sf m}_n< \frac{1}{\omega_n R^n}\int_0^R e^{\alpha_n \tilde w_a^\frac{n}{n-1}}{\rm d}{\sf m}_n\leq {\sf MT}_n,$$
	 which is a contradiction, showing the validity of \eqref{gradiens=2=MT-1-dim}.  
	 
	 Combining relations \eqref{MT-gradiens=1} and \eqref{gradiens=2=MT-1-dim}, it follows that 
	 $$\|\nabla u\|_{L^n(\Omega,{\sf m})}={\sf AVR}(X)^\frac{1}{n}\|(u^*)'\|_{L^n([0,R),{\sf m}_n)},$$
	 which means that we have equality in the P\'olya--Szeg\H o inequality \eqref{Polya-Szego} for the function $u$. Accordingly, it follows that the space $(X,{\sf d}, {\sf m})$ is an $n$-volume cone with a tip $x_0\in X$. 
	 
	 To prove that $u$ is $x_0$-radial, we show that $(u^*)'\neq 0$ on $(0,R)$. To see this, we recall the definition of ${\sf MT}_n$, which is the supremum in \eqref{supremum-constant-MT} for $\alpha=\alpha_n.$ Due to \eqref{gradiens=2=MT-1-dim} and \eqref{MT-tilde-egyenlet}, ${\sf MT}_n$ is achieved by the function $\tilde u$; in particular, the corresponding Euler--Lagrange equation reads as 
	\begin{equation}\label{Euler-Lag}
			-(|\tilde u'(\rho)|^{n-2}\tilde u'(\rho)\rho^{n-1})'=c\tilde u^\frac{1}{n-1}e^{\alpha_n \tilde u(\rho)^\frac{n}{n-1}}\rho^{n-1},\ \ \rho\in (0,R),
	\end{equation}
	 for some $c>0.$ We stress that $u^*$ is non-negative and non-increasing, so  $\tilde u$. Thus $\tilde u'\leq 0$ a.e.\ in $(0,R)$, and if $$w(\rho)=|\tilde u'(\rho)|^{n-2}\tilde u'(\rho)\rho^{n-1},$$ one has that $w\leq 0.$ By \eqref{Euler-Lag} we also have that $w'\leq 0$, i.e., $w$ is non-positive and non-increasing in the interval $(0,R)$. Let us assume by contradiction that there exists $r_0\in (0,R)$ such that $w(r_0)=0$. Since $w(0)=0$, it follows by monotonicity reason that $w\equiv 0$ in $(0,r_0)$, i.e., $\tilde u'\equiv 0$ in $(0,r_0)$. In particular, by \eqref{Euler-Lag} it follows that $\tilde u\equiv 0$, so  $u^*\equiv 0$ in $(0,r_0)$. Therefore, by the definition of $u^*$, it follows that $u^*\equiv 0$ on the whole interval $(0,R)$, which implies that $u\equiv 0$  on $\Omega$, contradicting the fact that $u$ is non-zero. Therefore, we necessarily have that $w<0$ in $(0,R)$, which implies that $\tilde u'<0$, thus $(u^*)'<0$ in $(0,R)$. Now, \eqref{u=u^*} implies that $u(x)=u^*\left({\sf AVR}(X)^\frac{1}{n}{\sf d}(x_0,x)\right)$ for ${\sf m}$-a.e.\ $x\in B_{x_0}({\sf AVR}(X)^{-\frac{1}{n}}R)$, thus $u$ is $x_0$-radial.  
	 
	 It remains to prove that $\Omega$ is  a ball. To do this, we know that $\{u>0\}\subset \Omega$; in fact in this inclusion we have equality ${\sf m}$-a.e. Indeed, since $(u^*)'<0$ in $(0,R)$, due to equation \eqref{Euler-Lag}, one has that $\{u^*>0\}=(0,R)$. Thus, by definition of $u^*$ and the latter relation, it follows that 
	 $${\sf m}(\{u>0\})={\sf m}_n(\{u^*>0\})={\sf m}_n(0,R)=\omega_n R^n={\sf m}(\Omega),$$ which implies that $\{u>0\}$ coincides  $ \Omega$ up to an ${\sf m}$-negligible set. Since $u$ is $x_0$-radial, it follows that $\Omega=B_{x_0}({\sf AVR}(X)^{-1/n}R)$ up to an ${\sf m}$-negligible set.

\textit{\underline{Step 4: proof of ${\rm (ii)}$.}}	 Since $u_0:[0,1]\to \mathbb R_+$ is the profile function of the extremizer in the Carleson--Chang problem, we have that 
$\|u_0'\|_{L^n([0,1),{\sf m}_n)}= 1$ and
\begin{equation}\label{extremal-Carleson-Chang}
	\frac{1}{\omega_n}\int_0^1 e^{\alpha_n u_0^\frac{n}{n-1}}{\rm d}{\sf m}_n= {\sf MT}_n.
\end{equation}
Let $R>0$ be  arbitrarily fixed.  Since $X$ is an $n$-volume cone with a tip $x_0\in X$, i.e. ${\sf m}(B_{x_0}(r))={\sf AVR}(X){\omega_nr^n}$ for all $r>0,$ see \eqref{volume-cone}, and 
$$u(x)={\sf AVR}(X)^{-1/n}u_0({\sf AVR}(X)^{1/n}R^{-1}{\sf d}(x_0,x)),\ \ x\in \Omega=B_{x_0}({\sf AVR}(X)^{-1/n}R),$$
a similar calculation as in \eqref{gradiens-N-cone} shows that
\begin{align}\label{gradiens-N-cone-Carleson}
	\nonumber	\|\nabla u\|^n_{L^n(X,{\sf m})}&=\int_X |\nabla u|^n{\rm d}{\sf m}=R^{-n}\int_{B_{x_0}({\sf AVR}(X)^{-\frac{1}{n}}R)} |u_0'|^n\left({\sf AVR}(X)^{1/n}R^{-1}{\sf d}(x_0,x)\right){\rm d}{\sf m}(x)\\&\nonumber=R^{-n}\omega_n\left(|u_0'|^n(1)R^n	-{\sf AVR}(X)^{\frac{1}{n}+1}R^{-1}\ds\int_0^{{\sf AVR}(X)^{-\frac{1}{n}}R} (|u_0'|^n)'({\sf AVR}(X)^{\frac{1}{n}}R^{-1}t)t^n{\rm d}t\right)
	\\&=n\omega_n\ds \int_0^{1} |u_0'|^n(\rho)\rho^{n-1}{\rm d}\rho=\|u_0'\|^n_{L^n([0,1),{\sf m}_{n})}=1.
\end{align}
Moreover, since $u_0(1)=0$, by  Proposition \ref{proposition-co-area} and relation \eqref{extremal-Carleson-Chang}, one has that 
\begin{align*}
	\frac{1}{{\sf m}(\Omega)}\int_\Omega e^{ {\sf AVR}(X)^\frac{1}{n-1} \alpha_n u^\frac{n}{n-1}}{\rm d}{\sf m}&=\frac{1}{\omega_n R^n}\int_{B_{x_0}({\sf AVR}(X)^{-1/n}R)} e^{\alpha_n u_0^\frac{n}{n-1}({\sf AVR}(X)^{1/n}R^{-1}{\sf d}(x_0,\cdot))}{\rm d}{\sf m}\\&=1-\int_0^1\left(e^{\alpha_n u_0^\frac{n}{n-1}(\rho)}\right)'\rho^n{\rm d}\rho=n\int_0^1e^{\alpha_n u_0^\frac{n}{n-1}(\rho)}\rho^{n-1}{\rm d}\rho ={\sf MT}_n,
\end{align*}
which concludes the proof. 
\end{proof}


\begin{remark}\rm 
	(a) We emphasize that for a \textit{fixed} domain $\Omega\subset X$, the constant ${\sf AVR}(X)^\frac{1}{n-1}\alpha_n $ in \eqref{MT-metric-space-CD}   cannot be optimal (even in the Euclidean setting). Indeed, this fact can be easily proved by using the Moser function \eqref{Moser-function} concentrated at a point; I would like to thank Carlo Morpurgo for pointing out this phenomenon. 
	In fact, the sharpness of the above constant is obtained by exploring the validity of the  inequality for any \textit{large} domain; in particular, for metric balls $B_{x_0}(R)$ with $R\to \infty,$ see the blow-down limiting argument in the  proof of Theorem \ref{MT-theorem}. 
	
(b) In Theorem \ref{MT-theorem-CD}/(i) we cannot conclude that $u=u_0({\sf d}(x_0,\cdot)),$ where  $u_0$ is  the profile function of the extremizer in \eqref{supremum-constant-MT} for $\Omega=B_0(1)$ and $\alpha=\alpha_n$, since the uniqueness of $u_0$ is not known. 
\end{remark}

\section{Final comments} \label{section-final}

We conclude the paper with some final comments.  
\subsection{Volume growth characterizes Sobolev inequalities on ${\sf CD}(0,N)$ spaces}
Let $(X,{\sf d},{\sf m})$ be an essentially non-branching  ${\sf CD}(0,N)$ metric measure space, $N>1$. Our results can be roughly summarized as stating that the  following statements are equivalent: 
\begin{itemize}
	\item Any  Sobolev-type inequality in \S \ref{section-5} \& \ref{section-borderline-0} holds on $(X,{\sf d},{\sf m})$ for some constant $C>0;$
	\item ${\sf AVR}(X) >0.$
\end{itemize}
 Indeed, just to exemplify the above equivalence, if the Nash  inequality \eqref{Nash-metric-space}
holds for some $C>0$, then Theorem \ref{Nash-theorem} implies that $	{\sf AVR}(X)\geq \left({{\sf CL}_N}/{C}\right)^{N+2} >0.$
 Conversely, if ${\sf AVR}(X) >0$ holds, then Theorem \ref{Nash-theorem-CD} shows that the Nash inequality holds with the sharp constant $C:={\sf AVR}(X)^{-\frac{1}{N+2}} {\sf CL}_N>0.$ Similar arguments work for the other inequalities as well. 

\begin{remark}\rm 
  A similar non-collapsing characterization as above can be found on Riemannian manifolds for Sobolev inequalities (i.e., the Gagliardo--Nirenberg inequality \eqref{GN-1-metric-space} for $\theta=1$), see Coulhon and Saloff-Coste \cite{C-SC}. In the non-smooth setting, Tewodrose \cite{Tew} proved global weighted Sobolev inequalities 
  on non-compact ${\sf CD} (0, N)$
 spaces satisfying a suitable volume growth condition, where ${\sf AVR}(X) >0$   appears as a particular form. 
\end{remark}

\subsection{Volume growths versus Sobolev inequalities involving singularities.} Instead of the Gagliardo--Nirenberg--Sobolev inequality from Theorem \ref{main-theorem-unified}, we can consider a general Caffarelli--Kohn--Nirenberg-type inequality involving singular terms; for the sake of simplicity, we only consider  the  setting of \S \ref{section-unified}.   Without repeating the arguments, we roughly state a similar result for inequalities involving singular terms. 

For the parameters $\alpha,\beta,\gamma\in \mathbb R$, $q,r>0$, $N,p>1$  and $\theta\in (0,1]$, we assume that 
\begin{equation}\label{balance-condition-11}
	\frac{1}{q}-	\frac{\alpha}{N}=\theta\left(\frac{1}{p}-\frac{1+\beta}{N}\right)+(1-\theta)\left(\frac{1}{r}-\frac{\gamma}{N}\right),
\end{equation}
and 
\begin{equation}\label{balance-condition-3}
	(1-\theta)\gamma+\theta(1+\beta)-\alpha\neq 0.
\end{equation}
In addition, we  assume  the validity of a sharp Caffarelli--Kohn--Nirenberg  inequality on $\mathbb R_+$: 
\begin{equation}\label{CKN-1-dimensional-model-2}
	\left\|\frac{u}{t^\alpha}\right\|_{L^q(\mathbb R_+,{\sf m}_{N})}\leq {\sf K}_{\rm opt} 	\left\|\frac{u'}{t^\beta}\right\|_{L^p(\mathbb R_+,{\sf m}_{N})}^\theta 	\left\|\frac{u}{t^\gamma}\right\|_{L^r(\mathbb R_+,{\sf m}_{N})}^{1-\theta},\ \ \forall u\in C_0^\infty(\mathbb R_+),
\end{equation}
providing also an extremal function with suitable integrability properties, similar to \eqref{extremal-asymptotics}. We notice that \eqref{balance-condition-11} is the balance condition for \eqref{CKN-1-dimensional-model-2}, while if $\alpha=\beta=\gamma=0$, relation \eqref{balance-condition-11} reduces to \eqref{balance-condition} and \eqref{balance-condition-3} is automatically satisfied. 

We now consider a metric measure space $(X,{\sf d},{\sf m})$   
supporting the \textit{Caffarelli--Kohn--Nirenberg inequality}  for some ${\sf C}>0$, having the form  
\begin{equation}\label{CKN-metric-space}
	\left\|\frac{u}{\rho_{x_0}^\alpha}\right\|_{L^q(X,{\sf m})}\leq {\sf C} \left\|\frac{|\nabla u|_p}{\rho_{x_0}^\beta}\right\|_{L^p(X,{\sf m})}^\theta \left\|\frac{u}{\rho_{x_0}^\gamma}\right\|_{L^r(X,{\sf m})}^{1-\theta},\ \ \forall u\in  {\rm Lip}_c(X), 
\end{equation}
where $\rho_{x_0}={\sf d}(x_0,\cdot)$ for some $x_0\in X.$ 

Under these assumptions, if the limit $L_N(x_0)$ exists in \eqref{Limit-exists-N}, then one can prove, similarly to Theorem \ref{main-theorem-unified}, that
\begin{equation}\label{vol-gro-CKN}
	{\sf AVR}(X)\geq \left(\frac{{\sf K}_{\rm opt}}{\sf C}\right)^\frac{N}{(1-\theta)\gamma+\theta(1+\beta)-\alpha}.
\end{equation}
The above result provides a  simple proof of well-known results in the literature,  known mostly in Riemannian and Finsler manifolds with non-negative Ricci curvature, where the corresponding Barenblatt functions together with an involved  comparison of ODE and ODI are used, see e.g.\ 
\begin{itemize}
	\item do Carmo and Xia \cite{doCarmo-Xia} for $\theta=1$, $p=2$ and $q=\frac{2N}{N-2+2(\alpha-\beta)}$ with  $0\leq \beta<\frac{N-2}{2}$, $\beta\leq \alpha+1<\beta+1$, $N\in \mathbb N$ (Riemannian manifolds);
	\item Krist\'aly and Ohta \cite{Kristaly-Ohta} for $\theta=1$, $p=2$, $\beta=0$ and $q=\frac{2N}{N-2+2\alpha}$ with   $ \alpha\in [0,1)$, $N\in \mathbb N$ (Bishop--Gromov metric measure spaces \& Finsler manifolds); 
	\item  Tokura,  Adriano and  Xia \cite{TAX}, where the parameters verify  \eqref{balance-condition-11} and \eqref{balance-condition-3}  (Bishop--Gromov metric measure spaces). 
\end{itemize} 

\begin{remark}\rm \label{remark-degenerate}
	In the limit case when \eqref{balance-condition-3} fails, i.e., $(1-\theta)\gamma+\theta(1+\beta)-\alpha= 0,$ the volume growth \eqref{vol-gro-CKN} degenerates to  ${\sf C}\geq {\sf K}_{\rm opt}$. 
\end{remark}

\subsection{Borderline cases of the Caffarelli--Kohn--Nirenberg inequality \eqref{CKN-metric-space}} 
  We now discuss two borderline cases of \eqref{CKN-metric-space} on ${\sf CD}(0,N)$ metric measure spaces. 
  \subsubsection{Heisenberg--Pauli--Weyl uncertainty principle}
  Let $\theta=\frac{1}{2}$, $\alpha=\beta=0$, $\gamma=-1$ and $p=q=r=2$ in \eqref{CKN-metric-space} -- which clearly verify the balance condition \eqref{balance-condition-11}, -- i.e., 
	\begin{equation}\label{HPW-1}
			\left\|u\right\|_{L^2(X,{\sf m})}\leq {\sf C} \left\|\nabla u\right\|_{L^2(X,{\sf m})}^{1/2} \left\|\rho_{x_0}{u}\right\|_{L^2(X,{\sf m})}^{1/2},\ \ \forall u\in  {\rm Lip}_c(X).
	\end{equation}
		It is well-known that the Heisenberg--Pauli--Weyl uncertainty principle \eqref{HPW-1} holds on the 1-dimensional model metric measure cone  $(\mathbb R_+,|\cdot|,{\sf m}_{N})$ with the sharp constant ${\sf C}={\rm K}_{\rm opt}=\sqrt{2/N}$ and the (unique) extremal is the Gaussian $u_0(t)=e^{-t^2}$, $t\geq 0$ (up to multiplicative constant and scaling). 
	Han and Xu \cite{Han-Xu} proved that if $(X,{\sf d},{\sf m})$ is an essentially non-branching  ${\sf CD}(0,N)$ metric measure space
	 for some $N>1$,  the Heisenberg--Pauli--Weyl uncertainty principle \eqref{HPW-1} holds and  ${\sf C}=\sqrt{2/N}$ is sharp if and only if $X$ is an $N$-volume cone. The Riemannian and Finsler versions of this result have been initially proved in Krist\'aly \cite{Kristaly-JMPA} and Huang,  Krist\'aly and Zhao \cite{HKW}, respectively. Apart from the above rigidity result, we do not know the sharp constant of \eqref{HPW-1} on ${\sf CD}(0,N)$ spaces. 
If  \eqref{HPW-1} holds, the problem is that no volume growth can be obtained -- similarly as in \eqref{vol-gro-CKN} --  since the parameters fail to satisfy  \eqref{balance-condition-3}. The only available information is that if \eqref{HPW-1} holds, then ${\sf C}\geq \sqrt{2/N}$, cf.\ Remark \ref{remark-degenerate}. 
	
\subsubsection{Hardy inequality} Let $\theta=1$, $\alpha=1$, $\beta=0$ and $1<p=q<N$ in \eqref{CKN-metric-space}, i.e., 
\begin{equation}\label{Hardy-1}
	\left\|\frac{u}{\rho_{x_0}}\right\|_{L^p(X,{\sf m})}\leq {\sf C} \left\|\nabla u\right\|_{L^p(X,{\sf m})},\ \ \forall u\in  {\rm Lip}_c(X).
\end{equation}
Although in the 1-dimensional model  cone  $(\mathbb R_+,|\cdot|,{\sf m}_{N})$ the optimal constant is ${\sf C}={\rm K}_{\rm opt}=\frac{p}{N-p}$, \textit{no} extremal function exists;  therefore, the above machinery cannot be applied to establish sharp Hardy inequalities on ${\sf CD}(0,N)$ spaces (note also that the parameters do not satisfy \eqref{balance-condition-3}). By using the P\'olya--Szeg\H o inequality \eqref{Polya-Szego} and a suitable Hardy--Littlewood inequality on ${\sf CD}(0,N)$ spaces, 
the expected sharp constant in \eqref{Hardy-1} seems to be   ${\sf AVR}(X)^{-\frac{1}{N}}\frac{p}{N-p}$ whenever ${\sf AVR}(X)>0$.



\vspace{0.5cm}

\noindent \textbf{Acknowledgments.} I would like to thank Michel Ledoux for raising the question on the sharp Nash and Gagliardo--Nirenberg--Sobolev inequalities on generic Riemannian manifolds and metric measure  spaces, as well as Carlo Morpurgo for discussions about Moser--Trudinger inequalities. 

\end{document}